\documentclass[a4paper,11pt]{article}
\usepackage[latin1]{inputenc}
\usepackage[english]{babel}
\usepackage{amsmath}
\usepackage{amsfonts}
\usepackage{amssymb}
\usepackage{epsfig}
\usepackage{amsopn}
\usepackage{amsthm}
\usepackage{color}
\usepackage{graphicx}
\usepackage{enumerate}
\newtheorem{theorem}{Theorem}[section]

\newtheorem{lemma}[theorem]{Lemma}
\newtheorem{proposition}[theorem]{Proposition}

\newtheorem{remark}[theorem]{Remark}
\newtheorem*{theorem*}{Theorem}
\newtheorem*{lemma*}{Lemma}
\newtheorem*{remark*}{Remark}
\newtheorem*{definition*}{Definition}
\newtheorem*{proposition*}{Proposition}
\newtheorem*{corollary*}{Corollary}
\numberwithin{equation}{section}
\newcommand{\real}{\mathbb{R}}

\let\ced=\c         % cedilla
\def\a{\alpha}
\def\b{\beta}

\def\e{\varepsilon}        % Also, \varepsilon
\def\cp{{\cal P}}

\def\qed{\,\unskip\kern 6pt \penalty 500
\raise -2pt\hbox{\vrule \vbox to8pt{\hrule width 6pt
\vfill\hrule}\vrule}\par}
\definecolor{darkblue}{rgb}{0.05, .05, .65}
\definecolor{darkgreen}{rgb}{0.1, .65, .1}
\definecolor{darkred}{rgb}{0.8,0,0}
\newcommand{\beqn}{\begin{equation}}
\newcommand{\eeqn}{\end{equation}}
\newcommand{\bear}{\begin{eqnarray}}
\newcommand{\eear}{\end{eqnarray}}
\newcommand{\bean}{\begin{eqnarray*}}
\newcommand{\eean}{\end{eqnarray*}}
\begin{document}

\title{\huge \bf Instantaneous shrinking and single point extinction for viscous Hamilton-Jacobi equations with fast diffusion}

\author{
\Large Razvan Gabriel Iagar\,\footnote{Instituto de Ciencias
Matem\'aticas (ICMAT), Nicolas Cabrera 13-15, Campus de Cantoblanco,
Madrid, Spain, \textit{e-mail:}
razvan.iagar@icmat.es},\footnote{Institute of Mathematics of the
Romanian Academy, P.O. Box 1-764, RO-014700, Bucharest, Romania.}
\\[4pt] \Large Philippe Lauren\c cot\,\footnote{Institut de
Math\'ematiques de Toulouse, UMR~5219, Universit\'e de
Toulouse, CNRS, F--31062 Toulouse Cedex 9, France. \textit{e-mail:}
Philippe.Laurencot@math.univ-toulouse.fr}\\ [4pt]
\Large Christian Stinner\, \footnote{Technische Universit\"at Kaiserslautern, Felix-Klein-Zentrum f\"ur Mathematik, Paul-Ehrlich-Str. 31, D-67663 Kaiserslautern, Germany, \textit{e-mail:} stinner@mathematik.uni-kl.de} }
\date{\today}
\maketitle

%%%%%%%%%%%%%%%%%%%%
%%%%%%%%%%%%%%%%%%%%
\begin{abstract}
For a large class of non-negative initial data, the solutions to the
quasilinear viscous Hamilton-Jacobi equation $\partial_t u-\Delta_p
u+|\nabla u|^q=0$ in $(0,\infty)\times\real^N$ are known to vanish
identically after a finite time when $2N/(N+1)<p\leq2$ and
$q\in(0,p-1)$. Further properties of this extinction phenomenon are
established herein: \emph{instantaneous shrinking} of the support is
shown to take place if the initial condition $u_0$ decays
sufficiently rapidly as $|x|\to\infty$, that is, for each $t>0$, the
positivity set of $u(t)$ is a bounded subset of $\real^N$ even if
$u_0>0$ in $\real^N$. This decay condition on $u_0$ is also shown to
be optimal by proving that the positivity set of any solution
emanating from a positive initial condition decaying at a slower
rate as $|x|\to\infty$ is the whole $\real^N$ for all times. The
time evolution of the positivity set is also studied: on the one
hand, it is included in a fixed ball for all times if it is
initially bounded (\emph{localization}). On the other hand, it
converges to a single point at the extinction time for a class of
radially symmetric initial data, a phenomenon referred to as
\emph{single point extinction}. This behavior is in sharp contrast
with what happens when $q$ ranges in $[p-1,p/2)$ and $p\in
(2N/(N+1),2]$ for which we show \emph{complete extinction}.
Instantaneous shrinking and single point extinction take place in
particular for the semilinear viscous Hamilton-Jacobi equation when
$p=2$ and $q\in (0,1)$ and seem to have remained unnoticed.
\end{abstract}
%%%%%%%%%%%%%%%%%%%%
%%%%%%%%%%%%%%%%%%%%

\vspace{1.0 cm}

%%%%%%%%%%%%%%%%%%%%
%%%%%%%%%%%%%%%%%%%%

\medskip

%%%%%%%%%%%%%%%%%%%%

%%%%%%%%%%%%%%%%%%%%
%%%%%%%%%%%%%%%%%%%%
\section{Introduction and results}\label{sec1}
%%%%%%%%%%%%%%%%%%%%
%%%%%%%%%%%%%%%%%%%%

We perform a detailed study of the finite time
extinction phenomenon for a class of diffusion equations with a
gradient absorption term, of the form
\begin{equation}\label{eq1}
\partial_tu-\Delta_pu+|\nabla u|^q=0, \quad u=u(t,x), \ (t,x)\in
(0,\infty)\times\real^N,
\end{equation}
where, as usual,
$$
\Delta_pu={\rm div}(|\nabla u|^{p-2}\nabla u)\ ,
$$
supplemented with the  initial condition
\begin{equation}\label{eq2}
u(0)=u_0, \qquad x\in\real^N.
\end{equation}
Throughout the paper we assume that
\begin{equation}
u_0\in W^{1,\infty}(\real^N)\ , \qquad u_0\ge 0\ , \ u_0\not\equiv 0\ . \label{hypIC}
\end{equation}
 The range of exponents under consideration is
\begin{equation}\label{exp}
p_c:=\frac{2N}{N+1}<p\leq2, \qquad 0<q<p-1,
\end{equation}
in which it is already known that extinction in finite time takes
place for initial data decaying sufficiently rapidly as $|x|\to\infty$, that is: there exists $T_e\in(0,\infty)$ such that
$u(T_e,x)=0$ for any $x\in\real^N$, but $\|u(t)\|_{\infty}>0$ for
any $t\in(0,T_e)$. The time $T_e$ is usually referred to as \emph{the extinction time} of the solution $u$. Let us notice at this point that the range of exponents includes both the \emph{semilinear case} $p=2$ (with $q\in(0,1)$), and the \emph{singular diffusion case} $p\in (p_c,2)$ (with
$q\in (0,p-1)$). These two diffusion operators usually depart strongly in their qualitative properties, but in the range \eqref{exp}, the gradient absorption term is dominating the evolution, thus explaining the similarity of the results for the linear and singular diffusions.

The main feature concerning equations such as \eqref{eq1} is the
\emph{competition} between the two terms in the equation: a
diffusion one and an absorption one, in form of a gradient term. As
the properties of the diffusion equation and of the
Hamilton-Jacobi equation (without diffusion) are very different, it
is of interest to study the effects of their merging in the
equation, depending on the relative positions of the exponents $p$
and $q$.

For the semilinear case $p=2$, a number of results are available by
now, due to the possibility of using semigroup theory or linear
techniques. Thus, it has been shown that there appear two critical
values for the exponent $q$, namely $q=q_*:=(N+2)/(N+1)$ and $q=1$.
The qualitative theory, including the large time behavior, is now
well understood for exponents $q>1$ after the series of works
\cite{ATU04, BSW02, BKaL04, BL99, BVD13, BGK04, GL07, GGK03, Gi05}.
In this range $q>1$, the diffusion has an important influence on the
evolution: either completely dominating, when $q>(N+2)/(N+1)$,
leading to asymptotic simplification, or having a similar effect to
the Hamilton-Jacobi part for $q\in(1,q_*]$, leading to a resonant,
logarithmic-type behavior for $q=q_*$ \cite{GL07}, or a behavior
driven by \emph{very singular solutions} for $q\in (1,q_*)$
\cite{BKaL04}. Much less is known for the complementary range of
$q$, that is, $0<q\leq1$, where the Hamilton-Jacobi term starts to
have a very strong influence on the dynamics. The limit exponent
$q=1$ is highly critical and not yet fully understood though optimal
temporal decay estimates are established in \cite{BRV97}, while
extinction in finite time has been shown for $q\in (0,1)$, see
\cite{BLS01, BLSS02, Gi05}. Still, a complete understanding of the
extinction phenomenon is missing and requires a deeper study. So
far, the class of initial data for which finite time extinction
takes place has not yet been identified and more qualitative
information on the behavior of the support of the solution and on
the rate and shape of the extinction are still lacking. The purpose
of the present work is to shed some further light on the extinction
phenomenon for \eqref{eq1} and provide more detailed information on
the above mentioned issues, not only in the semilinear case $p=2$
but also when the diffusion is nonlinear and singular corresponding
to $p\in (p_c,2)$.

Indeed, considering the quasilinear diffusion operator $\Delta_p$ is
a natural nonlinear generalization, but due to the fact that linear
techniques are not available anymore, its study is more involved and
results were obtained only recently. We are interested in the fast
diffusion case $p\in (p_c,2)$, for which the qualitative theory is
developed starting from \cite{IL12}, where all exponents $q>0$ are
considered. In particular, two critical exponents are identified in
\cite{IL12}: $q=q_*:=p-N/(N+1)$ and $q=p/2$. These critical values
limit ranges of parameters with different behaviors: diffusion
dominates for $q>q_*$, while there is a balance between diffusion
and absorption for $q\in [p/2,q_*]$ leading to logarithmic decay for
$q=q_*$, algebraic decay for $q\in(p/2,q_*)$, and exponential decay
for $q=p/2$. Finally, finite time extinction occurs for $0<q<p/2$
and this is the range of the parameter $q$ we are interested in. We
actually perform a deeper study of the extinction range $0<q<p/2$
with $p\in (p_c,2)$ which reveals very interesting and surprising
features having not been observed before, as far as we know. We
mention at this point that we restrict the analysis to $p>p_c$ as
there is a competition in this range between the diffusion term
which aims at positivity and the gradient absorption term which is
the driving mechanism of extinction. We left aside the critical case
$p=p_c$ which is trickier to handle, as well as the case $p\in
(1,p_c)$, for which a different competition takes place. Indeed,
finite time extinction is also known to take place for the diffusion
equation without the gradient absorption term when $p\in (1,p_c)$,
and there is then a competition between two extinction mechanisms
stemming from the diffusion and the gradient absorption,
respectively.

Before describing more precisely our results, let us recall that the finite time extinction phenomenon has already been observed as the outcome of a competition between diffusion and absorption effects, in particular for another important diffusive model, the porous medium equation with zero order absorption
\begin{equation}\label{PMEA}
\partial_tu-\Delta u^m+u^q=0, \quad (t,x)\in(0,\infty)\times\real^N,
\end{equation}
with $m>0$ and $q\in (0,1)$. A striking feature is the
\emph{instantaneous shrinking of the support} for non-negative solutions, that is, the solution $u(t)$ to \eqref{PMEA} at time $t$ is compactly supported for all $t>0$ even though the initial condition $u(0)$ is positive in $\real^N$. This phenomenon was first noticed in \cite{EK} for $m=1$ and later extended to \eqref{PMEA} and its variants (including variable coefficients in front of the absorption term and/or an additional convection term) for $q\in (0,1)$ and $m>q$ in \cite{Ab98, BU, EK, FH, GK90, HV89, K90, K93, KN92}, as well as to other equations such as
$$
\partial_t u - \Delta_p u + |u|^{q-1}u = 0, \quad (t,x)\in(0,\infty)\times\real^N,
$$
for suitable ranges of the parameters $p$ and $q$, see \cite{De08,
KS96} and the references therein. More precise information on the
behavior near the extinction time are available for \eqref{PMEA} in
one space dimension $N=1$ when $0<m\le 1$ and $0<q<1$, \cite{FGV, FV, HV92}. It is shown that the case $q=m$ is critical and the
mechanism of extinction is different whether $m<q<1$ or $0<q<m$: in
the former, \emph{simultaneous or complete extinction} occurs, that
is, the solution is positive everywhere in $\real^N$ prior to the
extinction time and vanishes identically at the extinction time. In
the latter, \emph{single point extinction} takes place as shown in
\cite{FV}, that is, the positivity set of $u(t)$ shrinks to a point
as $t$ approaches the extinction time. The limit case $q=m$ is
simpler and explicit, being the only case studied also for $N>1$
\cite{DPS}.

The purpose of this paper is to investigate the occurrence of the
above mentioned phenomena for Eq.~\eqref{eq1}. Concerning
instantaneous shrinking, we show that it takes place in the range
$q\in (0,p-1)$, but only for initial conditions which decay
sufficiently fast as $|x|\to\infty$. On the contrary, for
positive initial data with a slow decay as $|x|\to\infty$, the
solution is positive everywhere in $\real^N$ for all times. This is
in sharp contrast with the situation for \eqref{PMEA}, where it is
sufficient that the initial condition decays to zero as
$|x|\to\infty$ for instantaneous shrinking to take place \cite{Ab98,
BU, EK}. The occurrence of instantaneous shrinking in \eqref{eq1}
thus not only depends on the parameters $p$ and $q$ but also on the
shape of the initial condition. Coming back to complete extinction,
we also show that this is the generic behavior when $q\in
[p-1,p/2)$, while we identify a class of initial data for which
single point extinction occurs when $q\in(0,p-1)$.

\bigskip

\noindent \textbf{Main results}. We denote in the sequel the (spatial) positivity set $\mathcal{P}(t)$ at time $t\geq0$ of a solution $u$ to \eqref{eq1}-\eqref{eq2} by
\begin{equation}\label{pos.set}
\mathcal{P}(t):=\{x\in\real^N: u(t,x)>0\}.
\end{equation}
We begin with some features of the time evolution of the positivity set according to the decay of $u_0$ at infinity.

%%%%%%%%%%%%%%%%%%%%
\begin{theorem}[Instantaneous shrinking and localization]\label{th.inst}
Let $u$ be a solution to the Cau\-chy problem \eqref{eq1}-\eqref{eq2} with an initial condition $u_0$ satisfying \eqref{hypIC} and
\begin{equation}\label{init.decay}
u_0(x)\leq C(1+|x|)^{-\theta}, \qquad x\in\real^N, \qquad
\theta>\frac{q}{1-q}
\end{equation}
for some $C>0$, and exponents $p$, $q$ as in \eqref{exp}. Then:
\begin{itemize}
\item[(i)] Instantaneous shrinking: for any $t>0$, $\mathcal{P}(t)$ is a bounded subset of $\real^N$.
\item[(ii)] Localization: for any $\tau>0$ there exists $\varrho_\tau>0$ such that $\mathcal{P}(t)\subseteq B(0,\varrho_\tau)$ for all $t\ge \tau$.
\item[(iii)] Extinction: there is $T_e>0$ such that $u(t)\equiv 0$ for all $t\ge T_e$.
\end{itemize}
\end{theorem}
%%%%%%%%%%%%%%%%%%%%

In other words, the dynamics forces the support of the solution to
become compact immediately (at any time $t>0$) even if $u_0$ is
positive in $\real^N$, that is, $\mathcal{P}(0)=\real^N$. It then
remains confined inside a ball for $t\ge \tau>0$, the radius of the
ball depending only on $\tau$. Two steps are needed to prove
Theorem~\ref{th.inst}: we first construct a supersolution to
\eqref{eq1} on $(0,t_0)\times\real^N$ for a sufficiently small time
$t_0$ which is positive for $t=0$ but has compact support for $t\in
(0,t_0)$. Due to the gradient absorption term, it does not seem to
be possible to adapt the approach used for \eqref{PMEA} and the
supersolution has to be constructed in a different way. The second
step is to establish that solutions to \eqref{eq1}-\eqref{eq2}
emanating from compactly supported initial data enjoy the
localization property, that is, their support stays forever in a
fixed ball of $\real^N$. Combining these two steps provides the
first two assertions of Theorem~\ref{th.inst}, the last assertion
being a straightforward consequence of the compactness of the
support for positive times, \cite[Corollary~9.1]{Gi05}, and
\cite[Theorem~1.2(iii)]{IL12}.

Theorem~\ref{th.inst} turns out to be false under a sole condition of decay to zero at infinity of $u_0$ and the following result shows a strikingly different behavior for initial data with a sufficiently slow spatial decay at infinity.

%%%%%%%%%%%%%%%%%%%%
\begin{theorem}[Non-extinction and non-localization]\label{th.noninst}
Let $u$ be a solution to the Cauchy problem \eqref{eq1}-\eqref{eq2} with an initial condition $u_0$ satisfying \eqref{hypIC} and
\begin{equation}\label{init.decay2}
\lim\limits_{|x|\to\infty}|x|^{q/(1-q)}u_0(x)=\infty, \qquad
u_0(x)>0 \ {\rm for \ any } \ x\in\real^N,
\end{equation}
and exponents $p$, $q$ as in \eqref{exp}. Then
$$
\mathcal{P}(t)=\real^N \qquad {\rm for \ any} \ t>0.
$$
\end{theorem}
%%%%%%%%%%%%%%%%%%%%

Theorem~\ref{th.noninst} means that, when the initial condition has a sufficiently fat tail at infinity, it has enough mass to remain positive everywhere for all times in spite of the dominating absorption effect. When $p=2$ and $q\in (0,1)$ it is established in \cite{BLSS02} with the help of suitable subsolutions and we extend here this approach to the whole range \eqref{exp}.

A consequence of Theorem~\ref{th.noninst} is that the decay condition \eqref{init.decay} on $u_0$ is optimal for instantaneous shrinking to take place when the parameters $p$ and $q$ satisfy \eqref{exp}. The optimality of the range \eqref{exp} of the exponents $p$ and $q$ requires a different argument: we show in Proposition~\ref{prop.nonloc} that, when $p\in (p_c,2)$ and $q\in [p-1,p/2)$, only complete extinction takes place,  that is,
\begin{equation*}
\mathcal{P}(t)=\real^{N} \qquad {\rm for \ any} \ t\in (0,T_e).
\end{equation*}
the finiteness of the extinction time $T_e$ for that range being
proved in \cite{IL12} for initial data decaying suffciently rapidly
at infinity.

Returning to finite time extinction, we are able to improve
Theorem~\ref{th.inst} as well as \cite[Theorem~1.2(iii)]{IL12},
showing that solutions vanish after a finite time even in the limit
case for the decay $\theta=q/(1-q)$ in \eqref{init.decay}, which is
excluded in Theorem~\ref{th.inst}. However, no information on the
evolution of the positivity set is provided.

%%%%%%%%%%%%%%%%%%%%
\begin{theorem}[Improved finite time extinction]\label{th.improv}
Let $u$ be a solution to the Cauchy problem \eqref{eq1}-\eqref{eq2},
with an initial condition $u_0$ satisfying \eqref{hypIC} and
\begin{equation}\label{init.decay3}
u_0(x)\leq C_0(1+|x|)^{-q/(1-q)}, \qquad x\in\real^N,
\end{equation}
for some $C_0>0$ and $p$, $q$ as in \eqref{exp}. Then extinction in
finite time takes place: there exists $T_e\in(0,\infty)$ such that $u(T_e,x)=0$ for any $x\in\real^N$, but $\|u(t)\|_{\infty}>0$ for any $t\in(0,T_e)$.
\end{theorem}
%%%%%%%%%%%%%%%%%%%%

Theorem~\ref{th.improv} is proved in \cite{BLSS02} when $p=2$ and $q\in (0,1)$ and we extend it here to the whole range \eqref{exp}. Its proof relies on the construction of self-similar supersolutions vanishing identically after a finite time.

We now delve deeper in the extinction mechanism and aim at studying how extinction takes place. Let $u$ be a solution to \eqref{eq1} and
$T_e\in(0,\infty)$ its extinction time, assuming that $u$ vanishes
in finite time. Recalling the definition \eqref{pos.set} of the positivity set $\mathcal{P}(t)$, we define \emph{the
extinction set} of $u$ by
\begin{equation*}
\mathcal{E}(u):=\left\{
\begin{array}{l}
x\in\real^N: {\rm there \ exist} \ \{x_n\ :\ n\ge 1\}\subset\real^N, \ \{t_n\ :\ n\ge 1\}\subset(0,T_e) \\
\ {\rm such \ that} \ x_n\to x, \ t_n\to T_e \ {\rm as} \ n\to\infty, \ u(t_n,x_n)>0 \ {\rm for \ all} \ n
\end{array}
\right\}\ .
\end{equation*}
We say that $u$ presents \emph{simultaneous or complete extinction} if $\mathcal{E}(u)=\real^N$ while it presents \emph{single point extinction} when $\mathcal{E}(u)$ is a singleton. Simultaneous extinction is the most
common phenomenon; for example, it occurs for the standard
subcritical fast diffusion equation (without absorption terms).
Here, it happens that the opposite and less standard phenomenon occurs. More precisely:

%%%%%%%%%%%%%%%%%%%%
\begin{theorem}[Single point extinction]\label{th.single}
Let $u$ be a solution to the Cauchy problem \eqref{eq1}-\eqref{eq2} with an initial condition $u_0$ satisfying \eqref{hypIC}, and exponents $p$, $q$ as in \eqref{exp}. Assume further that:
\begin{itemize}
\item[(a)] $u_0\in C^1(\real^N)$ is radially symmetric and radially non-increasing,
\item[(b)] $u_0$ is compactly supported in $B(0,R_0)$ for some $R_0>0$ and satisfies the following condition
\begin{equation}
u_0(x) \le \kappa_{p,q} |x-x_0|^\omega\ , \qquad x\in\real^N\,
\label{locvan}
\end{equation}
for all $x_0\in\partial B(0,R_0)$, with
\begin{equation}\label{const}
\kappa_{p,q} := \frac{p-1-q}{p-q} \left( \frac{p-1}{p-1-q} + N-1
\right)^{-1/(p-1-q)} \quad\mbox{ and }\quad \omega:=
\frac{p-q}{p-1-q}\ ,
\end{equation}
\item[(c)] and there exists $\delta_0>0$ such that
\begin{equation}\label{init.data4}
\left| \nabla u_0^{(p-q-1)/(p-q)}(x) \right| \ge \delta_0
|x|^{1/(p-1-q)}\ , \qquad x\in B(0,R_0)\ .
\end{equation}
\end{itemize}
Let $T_e\in(0,\infty)$ be the extinction time of $u$, which is
finite according to Theorem~\ref{th.inst}. Then, there exist
$\varrho_1>0$ and $\varrho_2>0$ such that
\begin{equation}\label{incl.pos}
B\left(0,\varrho_1(T_e-t)^{\sigma}\right)\subseteq\mathcal{P}(t)\subseteq
B\left(0,\varrho_2(T_e-t)^{\nu}\right) \qquad {\rm for \ any} \
t\in(T_e/2,T_e),
\end{equation}
where
$$
\sigma:=\frac{p-q-1}{(p-q)(1-q)}, \qquad
\nu:=\frac{p(p-q-1)^2}{2(p-q)(p-2q)}.
$$
Consequently, $u$ presents \emph{single point extinction} at the
origin: $\mathcal{E}(u)=\{0\}$.
\end{theorem}
%%%%%%%%%%%%%%%%%%%%

As we shall see below the first inclusion in \eqref{incl.pos} holds true for any radially symmetric and radially non-increasing initial condition $u_0$ with compact support. The second inclusion requires the more restrictive conditions \eqref{locvan} and \eqref{init.data4} on $u_0$, the former guaranteeing that the positivity set of $u$ stays inside the ball $B(0,R_0)$.

As far as we know, Theorem~\ref{th.single} is the
\emph{first example} of single point extinction for equations
with gradient absorption. Single point extinction was already observed for the heat equation with zero order
absorption term in a bounded domain $\Omega$ of $\real^N$ with homogeneous Dirichlet boundary conditions
$$
\partial_t u =\Delta u -  u^q, \quad (t,x)\in (0,\infty)\times\Omega,
$$
when $q\in(0,1)$, the result being valid for a specific class of
inital conditions \cite{FH}. In space dimension $N=1$, similar
results are available for \eqref{PMEA} in the range $0<q<m<1$
\cite{FV}.

The proof of Theorem~\ref{th.single} is technically involved, following the general strategy used by Friedman and Herrero \cite{FH}, but with new and decisive contributions of the optimal gradient estimates established in \cite[Theorem~1.3]{IL12}. Notice also that we are able to drop any restriction of the type $\partial_t u(0,x)=\Delta u_0(x) - u_0(x)^q\geq 0$ on the initial data, as required in \cite{FH}.

\bigskip

\noindent \textbf{Organization of the paper.} After recalling the
well-posedness of \eqref{eq1}-\eqref{eq2} and a few properties of
solutions in Section~\ref{sec.wp}, we begin with proving the
instantaneous shrinking phenomenon, as stated in
Theorem~\ref{th.inst}, to which we devote Section~\ref{sec.inst}.
The proof of Theorem~\ref{th.inst} is completed in
Section~\ref{sec.loc}, where we prove the localization of
$\mathcal{P}(t)$, $t\ge 0$, for compactly supported initial data, as
well as some side results showing that the range \eqref{exp} is
optimal for localization to take place. The localization property
allows us to derive upper and lower bounds at the extinction time
which are gathered in Section~\ref{sec.ulbet}. We go on with the
proof of Theorem~\ref{th.noninst}, performed in
Section~\ref{sec.nonext}, and the proof of Theorem~\ref{th.improv},
done in the subsequent Section~\ref{sec.improv}. All these proofs
have in common the fact that they rely on the maximum principle,
used in suitable ways according to the case to be dealt with. In
particular, subsolutions and supersolutions of different kinds with
suitable behaviors are constructed along these sections. Finally, we
devote Section~\ref{sec.single} to the proof of
Theorem~\ref{th.single}, which is the most involved technically and
is further divided into several subsections. The paper ends with a
technical Appendix where we provide rigorous proofs for some
estimates and calculations performed only at a formal level in
Section~\ref{sec.single} for the simplicity of the reading.

\noindent\textbf{Notation.} We introduce the parabolic operator $\mathcal{L}$ defined by
\begin{equation}\label{eq.parop}
\mathcal{L} z := \partial_t z - \Delta_p z + |\nabla z|^q \quad\text{ in }\quad (0,\infty)\times\real^N\ .
\end{equation}
If $z$ is radially symmetric with respect to the space variable then, setting $r:=|x|$ and $z(t,r)=z(t,|x|)$, an alternative formula for $\mathcal{L} z$ is the following:
\begin{equation}\label{eq.paroprad}
\mathcal{L} z = \partial_t z - (p-1)|\partial_{r} z|^{p-2}\partial_r^2 z -\frac{N-1}{r} |\partial_{r} z|^{p-2}\partial_{r}z + |\partial_{r}z|^{q}\ .
\end{equation}

%%%%%%%%%%%%%%%%%%%%
%%%%%%%%%%%%%%%%%%%%
\section{Well-posedness}\label{sec.wp}
%%%%%%%%%%%%%%%%%%%%
%%%%%%%%%%%%%%%%%%%%

We collect in this section some properties of the Cauchy problem
\eqref{eq1}-\eqref{eq2} and its solutions. We first recall the
well-posedness of \eqref{eq1}-\eqref{eq2}, established in
\cite{GGK03} for $p=2$ and in \cite{IL12} for $p\in (1,2)$.

%%%%%%%%%%%%%%%%%%%%
\begin{proposition}\label{prop.wp}
Let $p\in (p_c,2]$ and $q>0$. Given an initial condition $u_0$ satisfying \eqref{hypIC}, there is a unique nonnegative viscosity solution
$$
u\in \mathcal{BC}([0,\infty)\times\mathbb{R}^N) \cap L^\infty(0,\infty;W^{1,\infty}(\real^N))
$$
to \eqref{eq1}-\eqref{eq2} which is also a weak solution if $p\in (p_c,2)$ and a classical solution if $p=2$. In addition, it satisfies
\begin{equation}
0 \le u(t,x) \le \|u_0\|_\infty\ , \qquad (t,x)\in
(0,\infty)\times\real^N\ . \label{zz4}
\end{equation}
\end{proposition}
%%%%%%%%%%%%%%%%%%%%

We next show that radial symmetry and radial monotonicity are both preserved by \eqref{eq1}.

%%%%%%%%%%%%%%%%%%%%
\begin{lemma}\label{lem.radial}
Let $u$ be a solution to \eqref{eq1} such that its initial condition $u_0$ satisfies \eqref{hypIC} and is radially symmetric and non-increasing (that is, for $y\in\real^N$ and $z\in\real^N$, there holds $u_0(y)\ge u_0(z)$ whenever $|y|\le |z|$). Then
$x\mapsto u(t,x)$ is radially symmetric and non-increasing for any $t>0$.
\end{lemma}
%%%%%%%%%%%%%%%%%%%%

\begin{proof}
The radial symmetry is immediate from the rotational invariance of Eq.~\eqref{eq1} and the uniqueness statement in Proposition~\ref{prop.wp}.

Consider next $y\in\real^N$ and $z\in\real^N$ such that $|y|<|z|$ and define $x_0:=(z-y)/2$, the hyperplane $\mathcal{H}:=\{ x\in\real^N\ :\ \langle x , x_0 \rangle>0\}$, and the functions $v_\pm(t,x) := u(t,x\pm x_0)$ for $(t,x)\in (0,\infty)\times\mathcal{H}$. On the one hand, for $x\in\partial\mathcal{H}$, there holds $\langle x,x_0\rangle=0$, so that
$$
|x+x_0|^2 = |x|^2 + |x_0|^2 = |x-x_0|^2,
$$
and the radial symmetry of $u$ entails that
$$
v_+(t,x) = u(t,x+x_0) = u(t,x-x_0) = v_-(t,x) \quad
\text{ for all } t>0 \;\;\text{ and }\;\;  x\in\partial\mathcal{H}\ .
$$
On the other hand, if $x\in\mathcal{H}$, then
$$
|x+x_0|^2 = |x|^2 + |x_0|^2 + 2 \langle x,x_0\rangle \ge |x|^2 + |x_0|^2 - 2 \langle x,x_0\rangle = |x-x_0|^2,
$$
and the radial monotonicity of $u_0$ implies that
$$
v_+(0,x) = u_0(x+x_0) \le u_0(x-x_0) = v_-(0,x).
$$
Since $v_+$ and $v_-$ both solve \eqref{eq1} in $(0,\infty)\times\mathcal{H}$, we infer from the comparison principle and the previous properties that $v_+ \le v_-$ in $(0,\infty)\times\mathcal{H}$. In particular, $(y+z)/2\in\mathcal{H}$ and we obtain
$$
u(t,y) = v_-\left( t,\frac{y+z}{2} \right) \ge v_+\left( t,\frac{y+z}{2} \right) = u(t,z)
$$
as claimed.
\end{proof}

We finally recall that extinction in finite time occurs for $p\in
(p_c,2]$ and $q\in (0,p/2)$ when the initial condition $u_0$ is
compactly supported.

%%%%%%%%%%%%%%%%%%%%
\begin{proposition}\label{prop.ext}
Let $p\in (p_c,2]$ and $q\in (0,p/2)$. Let $u_0$ be an initial
condition satisfying \eqref{hypIC} and denote the corresponding
solution to \eqref{eq1}-\eqref{eq2} by $u$. Assume further that
$u_0$ is compactly supported. There exists $T_e>0$ depending on $N$,
$p$, $q$, and $u_0$ such that
\begin{equation*}
\mathcal{P}(t)=\emptyset \ {\rm for} \ t\geq T_e, \qquad
\mathcal{P}(t) \neq\emptyset \ {\rm for} \
t\in[0,T_e).
\end{equation*}
\end{proposition}
%%%%%%%%%%%%%%%%%%%%

Proposition~\ref{prop.ext} is shown in \cite[Corollary~9.1]{Gi05} for $p=2$ and in \cite[Theorem 1.2(iii)]{IL12} for $p<2$. It is actually proved in the latter that finite time extinction takes place for a broader class of initial data, namely, if there exist $C_0>0$ and $Q>0$ such that
\begin{equation}\label{interm7}
u_0(x)\leq C_0|x|^{-(p-Q)/(Q-p+1)}, \qquad x\in\real^N,
\end{equation}
where
$$
Q=q \ {\rm if} \ q\in (q_1,p/2), \quad Q\in(q_1,p/2) \ {\rm if} \
q\in(0,q_1], \quad {\rm and} \
q_1:=\max\left\{p-1,\frac{N}{N+1}\right\}\ .
$$

%%%%%%%%%%%%%%%%%%%%
\begin{remark}\label{rem.improv}
It is worth pointing out here that Theorem~\ref{th.improv} includes \cite[Theorem 1.2(iii)]{IL12} when the range of $(p,q)$ is \eqref{exp}. Indeed, if $u_0$ satisfies \eqref{interm7} then it satisfies \eqref{init.decay3}.
\end{remark}
%%%%%%%%%%%%%%%%%%%%

%%%%%%%%%%%%%%%%%%%%
%%%%%%%%%%%%%%%%%%%%
\section{Instantaneous shrinking}\label{sec.inst}
%%%%%%%%%%%%%%%%%%%%
%%%%%%%%%%%%%%%%%%%%

In this section, we show that the phenomenon of instantaneous
shrinking takes place, thus proving the first assertion in
Theorem~\ref{th.inst}. Its proof will be completed in the next
section, which deals with the localization part, once the support is
known to be compact. More precisely, we show
here the following result.

%%%%%%%%%%%%%%%%%%%%
\begin{proposition}\label{prop.inst}
Let $u$ be a solution to the Cauchy problem \eqref{eq1}-\eqref{eq2} with an initial condition $u_0$ satisfying \eqref{hypIC} and \eqref{init.decay} for some $C>0$, and exponents $p$, $q$ as in \eqref{exp}. There
exists $t_0>0$ such that $\mathcal{P}(t)$ is bounded for
$t\in(0,t_0)$.
\end{proposition}
%%%%%%%%%%%%%%%%%%%%

\begin{proof}
We look for a supersolution to Eq.~\eqref{eq1} of the form
\begin{equation}\label{sigm}
\Sigma(t,x)=\left[\frac{A}{1+|x|^{\a}}-\eta(t)\right]_{+}^{\gamma},
\qquad (t,x)\in(0,\infty)\times\real^N,
\end{equation}
where $z_+:=\max\{z,0\}$ denotes the positive part of the real number $z$, the positive parameters $A>0$, $\a\in (0,1)$, $\gamma>1$, and the function $\eta$ being to be determined. For further use, we set
$$
r:=|x|, \quad y:=\frac{A}{1+r^{\a}}-\eta(t).
$$
Owing to the radial symmetry of $\Sigma$, it follows from \eqref{eq.paroprad} that
$$
\mathcal{L}\Sigma=\partial_{t}\Sigma-(p-1)|\partial_{r}\Sigma|^{p-2}\partial_{r}^2\Sigma-\frac{N-1}{r}|\partial_{r}\Sigma|^{p-2}\partial_{r}\Sigma+|\partial_{r}\Sigma|^{q}.
$$
We further require that $\eta(0)=0$ and that $\eta$ is
non-decreasing, that is, $\eta'\geq0$. In the previous notation, we
notice that
$$
\partial_t\Sigma(t,x)=-\gamma y_+^{\gamma-1}\eta'(t),
$$
$$
\partial_r\Sigma(t,x)=-A\a\gamma
y_+^{\gamma-1}\frac{r^{\a-1}}{(1+r^{\a})^2},
$$
and
\begin{equation*}
\begin{split}
\partial_{r}^2\Sigma(t,x) &=A^2\a^2\gamma(\gamma-1)y_+^{\gamma-2}\frac{r^{2\a-2}}{(1+r^{\a})^4}\\
&-A\a\gamma
y_+^{\gamma-1}\frac{r^{\a-2}}{(1+r^{\a})^3}\left[\a-1-(\a+1)r^{\a}\right],
\end{split}
\end{equation*}
whence, we get
\begin{equation*}
\begin{split}
\mathcal{L}\Sigma&=-\gamma
y_+^{\gamma-1}\eta'(t)+(A\a\gamma)^qy_+^{q(\gamma-1)}\frac{r^{q(\a-1)}}{(1+r^{\a})^{2q}}-(A\a\gamma)^{p-2}y_+^{(\gamma-1)(p-2)}\frac{r^{(\a-1)(p-2)}}{(1+r^{\a})^{2(p-2)}}\\
&\times\left[(p-1)A^2\a^2\gamma(\gamma-1)y_+^{\gamma-2}\frac{r^{2\a-2}}{(1+r^{\a})^4}-(p-1)A\a\gamma
y_+^{\gamma-1}\frac{r^{\a-2}}{(1+r^{\a})^3}\left(\a-1-(\a+1)r^{\a}\right)\right.\\
&\left.-A\a\gamma(N-1)y_+^{\gamma-1}\frac{r^{\a-2}}{(1+r^{\a})^2}\right]\\
&=(A\a\gamma)^qy_+^{q(\gamma-1)}\frac{r^{q(\a-1)}}{(1+r^{\a})^{2q}}-\gamma
y_+^{\gamma-1}\eta'(t)+(A\a\gamma)^{p-1}y_+^{(\gamma-1)(p-1)-1}\frac{r^{(\a-1)(p-1)-1}}{(1+r^{\a})^{2(p-1)}}\\
&\times\left[(N-1)y_++(p-1)y_+\frac{2\a-(1+\a)(1+r^{\a})}{1+r^{\a}}-A\a(\gamma-1)(p-1)\frac{r^{\a}}{(1+r^{\a})^2}\right]\\
&\geq(A\a\gamma)^qy_+^{q(\gamma-1)}\frac{r^{q(\a-1)}}{(1+r^{\a})^{2q}}-\gamma
y_+^{\gamma-1}\eta'(t)+(A\a\gamma)^{p-1}y_+^{(\gamma-1)(p-1)-1}\frac{r^{(\a-1)(p-1)-1}}{(1+r^{\a})^{2(p-1)}}\\
&\times\left[-(1+\a)(p-1)y_+-\frac{A\a(\gamma-1)(p-1)}{1+r^{\a}}\right],
\end{split}
\end{equation*}
or, equivalently,
\begin{equation}\label{interm1}
\begin{split}
\mathcal{L}\Sigma&\geq(A\a\gamma)^qy_+^{q(\gamma-1)}\frac{r^{q(\a-1)}}{(1+r^{\a})^{2q}}-\gamma
y_+^{\gamma-1}\eta'(t)\\&-(A\a\gamma)^{p-1}A\a(\gamma-1)(p-1)y_+^{(\gamma-1)(p-1)-1}\frac{r^{(\a-1)(p-1)-1}}{(1+r^{\a})^{2(p-1)+1}}\\
&-(1+\a)(p-1)(A\a\gamma)^{p-1}y_+^{(\gamma-1)(p-1)}\frac{r^{(\a-1)(p-1)-1}}{(1+r^{\a})^{2(p-1)}}.
\end{split}
\end{equation}
We now choose
$$
\gamma=\frac{p-q}{p-q-1}>1,
$$
so that
$$
\gamma-1=\frac{1}{p-1-q} \qquad {\rm and} \qquad
(\gamma-1)(p-1)-1=q(\gamma-1).
$$
Since $\gamma>1$ and $\eta\ge 0$, we notice that
\begin{equation}\label{interm2}
y_+\leq\frac{A}{1+r^{\a}}, \qquad {\rm and} \qquad
y_+^{\gamma-1}=y_+^{q(\gamma-1)}y_+^{(1-q)(\gamma-1)}\leq
y_+^{q(\gamma-1)}\frac{A^{(1-q)(\gamma-1)}}{(1+r^{\a})^{(1-q)(\gamma-1)}}.
\end{equation}
Consequently, owing to the fact that $\eta'(t)\geq0$ for all $t>0$
and plugging \eqref{interm2} into \eqref{interm1}, we deduce
\begin{equation}\label{zz6}
\begin{split}
\mathcal{L}\Sigma&\geq
y_+^{q(\gamma-1)}\left[(A\a\gamma)^q\frac{r^{q(\a-1)}}{(1+r^{\a})^{2q}}-\frac{A^{(1-q)(\gamma-1)}\gamma}{(1+r^{\a})^{(1-q)(\gamma-1)}}\eta'(t)\right.\\
&\left.-(p-1)(\a\gamma)^{p-1}A^p(1+\a\gamma)\frac{r^{(\a-1)(p-1)-1}}{(1+r^{\a})^{2(p-1)+1}}\right]\\
&\geq
y_+^{q(\gamma-1)}\left[\frac{(A\a\gamma)^q}{2}\frac{r^{q(\a-1)}}{(1+r^{\a})^{2q}}-\frac{A^{(1-q)(\gamma-1)}\gamma}{(1+r^{\a})^{(1-q)(\gamma-1)}}\eta'(t)\right.\\
&\left.+\frac{(A\a\gamma)^q}{2}\frac{r^{q(\a-1)}}{(1+r^{\a})^{2q}}-(p-1)(\a\gamma)^{p-1}A^p(1+\a\gamma)\frac{r^{(\a-1)(p-1)-1}}{(1+r^{\a})^{2(p-1)+1}}\right]\\
& =
y_+^{q(\gamma-1)}\frac{A^{(1-q)(\gamma-1)}\gamma}{(1+r^{\a})^{(1-q)(\gamma-1)}} S_1 +y_+^{q(\gamma-1)}\frac{(A\a\gamma)^q}{2}\frac{r^{(\a-1)(p-1)-1}}{(1+r^{\a})^{2(p-1)+1}} S_2\ ,
\end{split}
\end{equation}
where
$$
S_1 := \frac{(\a\gamma)^q}{2\gamma}\frac{A^{(q-1)\gamma+1}r^{q(\a-1)}}{(1+r^{\a})^{q(\gamma+1)+1-\gamma}}-\eta'(t)
$$
and
$$
S_2 := \frac{r^{1+(\a-1)(q-p+1)}}{(1+r^{\a})^{2(q-p+1)-1}}
-2(p-1)(\a\gamma)^{p-1-q}A^{p-q}(1+\a\gamma)
$$
Our goal is now to show that $\Sigma$ is a supersolution to the Cauchy problem \eqref{eq1}-\eqref{eq2} in $(0,t_0)\times \left( \real^N\setminus B(0,R) \right)$ for some $t_0>0$ small enough and $R>1$ sufficiently large. To this end, we estimate separately $S_1$ and $S_2$. On the one hand, since $2(p-1-q)>0$ by \eqref{exp},
\begin{equation}\label{interm3}
\begin{split}
S_2&=r^{1+(\a-1)(q-p+1)}(1+r^{\a})^{1+2(p-1-q)}-2(p-1)(\a\gamma)^{p-1-q}A^{p-q}(1+\a\gamma)\\
&\geq r^{(\a+1)(p-q)}-2(1+\a\gamma)(p-1)(\a\gamma)^{p-1-q}A^{p-q}\\
&\geq R^{(\a+1)(p-q)}-2(1+\a\gamma)(p-1)(\a\gamma)^{p-1-q}A^{p-q},
\end{split}
\end{equation}
provided $r>R$. On the other hand, for $r>R$ such that $y>0$,
$$
\left( \frac{r^{\a}}{1+r^{\a}} \right)^{q(\a-1)/\a} \geq1 \quad\text{ and }\quad \frac{1}{1+r^\a} \geq \frac{\eta(t)}{A},
$$
since $\a\in (0,1)$, hence
\begin{equation*}
\begin{split}
S_1&\geq\frac{(\a\gamma)^q}{2\gamma}A^{(q-1)\gamma+1} \left(\frac{1}{1+r^{\a}}\right)^{(\a(1+q\gamma-\gamma)+q)/\a}-\eta'(t)\\
&\geq\frac{(\a\gamma)^q}{2\gamma}A^{(q-1)\gamma+1}\left(\frac{\eta(t)}{A}\right)^{(\a(1+q\gamma-\gamma)+q)/\a}-\eta'(t),
\end{split}
\end{equation*}
provided that $\a(1+q\gamma-\gamma)+q>0$. Since $q+1<p \le 2$,
$$
1+q\gamma-\gamma = \frac{q(p-q)-1}{p-q-1} \le \frac{q(2-q)-1}{p-q-1} = - \frac{(1-q)^2}{p-q-1} < 0\ ,
$$
and the previous condition on $\alpha$ reads
\begin{equation*}
\alpha < \alpha_2 := \min \left\{ \frac{q}{\gamma(1-q)-1},1 \right\} \ .
\end{equation*}
The lower bound on $S_1$ then becomes
\begin{equation}\label{interm4}
S_1\geq\frac{(\a\gamma)^q}{2\gamma}A^{-q/\a}\eta(t)^{(\a(1+q\gamma-\gamma)+q)/\a}-\eta'(t).
\end{equation}
We thus choose $\eta$ as the solution to the differential
equation
\begin{equation}\label{interm5}
\eta'(t)=\frac{(\a\gamma)^q}{2\gamma}A^{-q/\a}\eta(t)^{(\a(1+q\gamma-\gamma)+q)/\a},
\qquad \eta(0)=0,
\end{equation}
which is possible when $(\a(1+q\gamma-\gamma)+q)/\a<1$. Taking into account the precise value of $\gamma$, this condition reads
$$
\a>\a_1:=\frac{p-1-q}{p-q}\frac{q}{1-q}=\frac{q}{\gamma(1-q)}.
$$
It is easy to check that $\a_1<\a_2$, so that, any $\a\in(\a_1,\a_2)$ satisfies the conditions we assumed up to now. Since $S_1\ge 0$ by \eqref{interm4} and \eqref{interm5}, we infer from \eqref{zz6} and \eqref{interm3} that $\Sigma$ is a supersolution to \eqref{eq1} in $(0,\infty)\times\real^N\setminus B(0,R)$, provided that
\begin{equation}\label{cond1}
R^{(\a+1)(p-q)}\geq2(1+\a\gamma)(p-1)(\a\gamma)^{p-1-q}A^{p-q} \ .
\end{equation}
It remains to check that $\Sigma$ is a supersolution also for the initial and boundary conditions, that is
\begin{equation}\label{comp1}
\Sigma(0,x)=\left(\frac{A}{1+r^{\a}}\right)^{\gamma}\geq u_0(x)
\quad {\rm for \ any} \ x\in\real^N\setminus B(0,R),
\end{equation}
and
\begin{equation}\label{comp2}
\Sigma(t,x)\geq u(t,x), \quad {\rm for \ any} \ t\in(0,t_0), \
|x|=R.
\end{equation}
In order to check \eqref{comp1}, we first readily notice that in the range given by \eqref{exp}
$$
\gamma=\frac{p-q}{p-1-q}>\frac{q}{1-q},
$$
so that, owing to \eqref{init.decay}, there exists $\theta'>0$ such that
$$
u_0(x)\leq C(1+|x|)^{-\theta'}, \quad
\frac{q}{1-q} =\gamma \alpha_1 <\theta'< \gamma \alpha_2 = \min \left\{ \frac{\gamma q}{\gamma(1-q)-1}, \gamma \right\}\ .
$$
Indeed, if the initial decay exponent $\theta$ satisfies $\theta<\gamma\a_2$, then we take $\theta'=\theta$, while, if $\theta \ge \gamma\a_2$, we anyway have
$$
u_0(x)\leq C(1+|x|)^{-\theta}\leq C(1+|x|)^{-\theta'}, \qquad
x\in\real^N,
$$
for any $\theta'\in(q/(1-q),\gamma\a_2)$. Then, we define
\begin{equation}\label{interm6}
\a:=\frac{\theta'}{\gamma}\in (\a_1,\a_2),
\end{equation}
and we have
\begin{equation*}
\begin{split}
u_0(x)&\leq
C(1+|x|)^{-\theta'}=\frac{C}{(1+r)^{\a\gamma}}\leq\frac{C}{2^{\gamma(\a-1)}(1+r^{\a})^{\gamma}}\\
&=\left(\frac{C^{1/\gamma}}{2^{\a-1}}\frac{1}{1+r^{\a}}\right)^{\gamma},
\end{split}
\end{equation*}
after using the elementary inequality
$$
(1+r)^{\a}\geq2^{\a-1}(1+r^{\a}).
$$
We thus derive the inequality in \eqref{comp1} by requiring that
\begin{equation}\label{cond2}
A\geq\frac{C^{1/\gamma}}{2^{\a-1}}.
\end{equation}
In order to establish \eqref{comp2}, we further prescribe the following condition:
\begin{equation}\label{cond3}
A>(1+R^{\a})\|u_0\|_{\infty}^{1/\gamma},
\end{equation}
with $\a$ already chosen in \eqref{interm6}. By a direct integration
of the differential equation \eqref{interm5}, we obtain
$$
\eta(t)=\left[\frac{(\a\gamma)^q(1-\b)}{2\gamma}A^{-q/\a}\right]^{1/(1-\b)}t^{1/(1-\b)},
\qquad \b:=\frac{\a(1+q\gamma-\gamma)+q}{\a}\in(0,1),
$$
for $t\ge 0$. Taking into account that $\eta(t)\to 0$ as $t\to 0$, there exists $t_0>0$ sufficiently small such that, for $t\in [0,t_0]$ and $x\in\real^N$, $|x|=R$,
$$
\left( \frac{A}{1+R^\a} - \eta(t) \right)^\gamma \ge \left(\frac{A}{1+R^{\a}}-\eta(t_0)\right)^{\gamma}>\|u_0\|_{\infty} \ge u(t,x),
$$
which implies \eqref{comp2}. It only remains to show that the
conditions \eqref{cond1}, \eqref{cond2} and \eqref{cond3} are
compatible. To this end, we first choose $R>0$ sufficiently large
such that it satisfies simultaneously the following estimates:
\begin{equation}\label{choice1}
\begin{split}
R^{\a} & >\frac{C^{1/\gamma}}{2^{\a-1}\|u_0\|_{\infty}^{1/\gamma}}-1, \\
R^{\a+1} &>2\left[2(p-1)(1+\a\gamma)(\a\gamma)^{p-1-q}\right]^{1/(p-q)}(1+R^{\a})\|u_0\|_{\infty}^{1/\gamma},
\end{split}
\end{equation}
then choose $A>0$ such that
\begin{equation}\label{choice2}
(1+R^{\a})\|u_0\|_{\infty}^{1/\gamma}<A<2(1+R^{\a})\|u_0\|_{\infty}^{1/\gamma}.
\end{equation}
It is immediate to check that these choices of $R$ and $A$ satisfy
\eqref{cond1}, \eqref{cond2}, and \eqref{cond3}. Thus, letting
$\gamma=(p-q)/(p-q-1)$, $\a$ as in \eqref{interm6}, $R$ as in
\eqref{choice1}, and $A$ as in \eqref{choice2}, the function $\Sigma$
introduced in \eqref{sigm} is a supersolution to the Cauchy problem
\eqref{eq1}-\eqref{eq2} in $(0,t_0)\times\real^N\setminus B(0,R)$.
By the comparison principle, we get
$$
u(t,x)\leq\Sigma(t,x), \qquad t\in (0,t_0), \ |x|\geq R.
$$
Consequently, for $t\in (0,t_0]$, we infer from the definition \eqref{sigm} of $\Sigma$ that $\Sigma(t,x)=0$ for $|x|^{\a}>A/\eta(t)-1$, so that $\mathcal{P}(t)$ is bounded.
\end{proof}

%%%%%%%%%%%%%%%%%%%%
%%%%%%%%%%%%%%%%%%%%
\section{Propagation of the support. Localization}\label{sec.loc}
%%%%%%%%%%%%%%%%%%%%
%%%%%%%%%%%%%%%%%%%%

In this section, we complete the proof of Theorem~\ref{th.inst} by showing the second statement, observing that it readily implies the third one by Proposition~\ref{prop.ext}. To this end we actually establish that, if $u_0$ vanishes sufficiently rapidly at $x_0\in\real^N$, then $u(t,x_0)=0$ for all $t\ge 0$.

%%%%%%%%%%%%%%%%%%%%
\begin{proposition}\label{prop.locvan}
Let $u_0$ be an initial condition satisfying \eqref{hypIC} and
denote the corresponding solution to \eqref{eq1}-\eqref{eq2} by $u$.
Assume further that $(p,q)$ satisfies \eqref{exp} and that there is $x_0\in\real^N$ such that \eqref{locvan} holds true, with constants $\kappa_{p,q}$ and $\omega$ given in \eqref{const}. Then $u(t,x_0)=0$ for all $t\ge 0$.
\end{proposition}
%%%%%%%%%%%%%%%%%%%%

\begin{proof}
Owing to the invariance by translation of \eqref{eq1} we may assume without loss of generality that $x_0=0$. Setting $\Sigma(x) := \kappa_{p,q} r^\omega$ and $r:=|x|$ for $x\in\real^N$, we infer from the radial symmetry of $\Sigma$ and \eqref{eq.paroprad} that
\begin{equation*}
\mathcal{L}\Sigma(x) = (\omega\kappa_{p,q})^q r^{q(\omega-1)} - (\omega\kappa_{p,q})^{p-1} \left[ (p-1)(\omega-1) + N-1 \right] r^{(\omega-1)(p-1)-1} \ .
\end{equation*}
Since
$$
q(\omega-1) = \frac{q}{p-1-q} = (\omega-1)(p-1)-1 \quad\text{ and
}\quad (p-1)(\omega-1) + N-1 = (\omega\kappa_{p,q})^{q+1-p}\ ,
$$
we end up with
$$
\mathcal{L}\Sigma(x) = (\omega\kappa_{p,q})^q r^{q(\omega-1)} \left[ 1 - (\omega\kappa_{p,q})^{p-1-q} (\omega\kappa_{p,q})^{q+1-p} \right]= 0\ .
$$
Consequently $\Sigma$ is a solution to \eqref{eq1} and we infer from \eqref{locvan} and the comparison principle that $u(t,x)\le \Sigma(x)$ for $(t,x)\in [0,\infty)\times\real^N$. In particular, $u(t,0)\le\Sigma(0)=0$  for $t\ge 0$ as claimed
\end{proof}

The localization property is then a straightforward consequence of Proposition~\ref{prop.locvan}.

%%%%%%%%%%%%%%%%%%%%
\begin{proposition}\label{prop.loc}
Let $u_0$ be an initial condition satisfying \eqref{hypIC} and denote the corresponding solution to \eqref{eq1}-\eqref{eq2} by $u$. Assume further that $(p,q)$ satisfies \eqref{exp} and that $u_0$ is compactly supported. There exists $R>0$ such that
$$
\mathcal{P}(t)\subseteq B(0,R), \qquad {\rm for \ any} \
t\ge 0\ .
$$
\end{proposition}
%%%%%%%%%%%%%%%%%%%%

\begin{proof}
Let $R_0>0$ be such that $\mathcal{P}(0)\subseteq B(0,R_0)$ and set
$R:=R_0 + \left( \|u_0\|_\infty \kappa_{p,q}^{-1}
\right)^{1/\omega}$ with $\kappa_{p,q}$ and $\omega$ defined in \eqref{const}. Consider $x_0\in \real^N$ such that
$|x_0| \ge R$. On the one hand, for $x\in
B(x_0,R-R_0)$,
$$
|x| \ge |x_0|-|x-x_0| \ge R-(R-R_0)=R_0
$$
and $u_0(x) = 0 \le \kappa_{p,q} |x-x_0|^\omega$. On the other hand, for $x\not\in B(x_0,R-R_0)$,
$$
u_0(x) \le \|u_0\|_\infty = \kappa_{p,q}
(R-R_0)^\omega \le \kappa_{p,q} |x-x_0|^\omega\ .
$$
We are then in a position to apply Proposition~\ref{prop.locvan} and conclude that $u(t,x_0)=0$ for all $t\ge 0$. Since $x_0$ is arbitrary in $\real^N\setminus B(0,R)$ we have shown that $\mathcal{P}(t)\subseteq B(0,R)$ for all $t\ge 0$ and the proof is complete.
\end{proof}

Theorem~\ref{th.inst} is now an immediate corollary of
Propositions~\ref{prop.ext}, \ref{prop.inst} and~\ref{prop.loc}.
Indeed, there is $t_0>0$ such that $\mathcal{P}(t)$ is bounded for
$t\in(0,t_0]$ by Proposition~\ref{prop.inst}.
Proposition~\ref{prop.ext} applied to $u(\cdot + t_0)$ then implies
the finite time extinction of $u$ while Proposition~\ref{prop.loc}
guarantees the localization property for $t\ge\tau$ and $\tau>0$.

\bigskip

Another consequence of Proposition~\ref{prop.locvan} is the
\emph{infinite waiting time} phenomenon, that is, the fact that the
positivity set might not expand with time.

%%%%%%%%%%%%%%%%%%%%
\begin{proposition}\label{prop.iwt}
Let $u_0$ be an initial condition satisfying \eqref{hypIC} and denote the corresponding solution to \eqref{eq1}-\eqref{eq2} by $u$. Assume further that $(p,q)$ satisfies \eqref{exp} and that $u_0$ is compactly supported in $B(0,R_0)$ for some $R_0>0$ and satisfies \eqref{locvan} for all $x_0\in\partial B(0,R_0)$. Then
$$
\mathcal{P}(t) \subseteq \mathcal{P}(0)\ , \qquad t\ge 0\ .
$$
\end{proposition}
%%%%%%%%%%%%%%%%%%%%

\begin{proof}
In view of Proposition~\ref{prop.locvan} it is sufficient to check that $u_0$ satisfies \eqref{locvan} for all $x_0\in\real^N$ such that $|x_0|>R_0$. Indeed, consider $x_0\in \real^N$ with $|x_0|>R_0$. Set $\xi := x_0/|x_0|$. Since \eqref{locvan} is satisfied for $R_0\xi$, there holds
$$
u_0(x) \le \kappa_{p,q} |x-R_0\xi|^{\omega}\ , \qquad x\in\real^N\ .
$$
Consider now $x\in B(0,R_0)$. Then $x=\langle x , \xi \rangle \xi + y$ and, since $|\langle x , \xi \rangle|\le |x|< R_0$,
$$
|x-R_0\xi|^2 = \left( R_0 - \langle x , \xi \rangle \right)^2 + |y|^2 \le \left( |x_0| - \langle x , \xi \rangle \right)^2 + |y|^2 = |x-x_0|^2\ .
$$
Consequently, $u_0(x) \le \kappa_{p,q} |x-x_0|^{\omega}$ for all $x\in B(0,R_0)$. This inequality being obviously true for $x\not\in B(0,R_0)$ since $u_0(x)=0$, we have thus shown that $u_0$ satisfies \eqref{locvan} for all $x_0\not\in B(0,R_0)$ and Proposition~\ref{prop.iwt} readily follows from Proposition~\ref{prop.locvan}.
\end{proof}

\bigskip

The last result of this section is devoted to the optimality of the range \eqref{exp} of the parameters $p$ and $q$ for instantaneous shrinking to take place.

%%%%%%%%%%%%%%%%%%%%
\begin{proposition}\label{prop.nonloc}
Consider an initial condition $u_0$ satisfying \eqref{hypIC} and denote the corresponding solution to \eqref{eq1}-\eqref{eq2} by $u$. Let $p\in (p_c,2)$ and assume that $u_0$ satisfies \eqref{interm7}. If $q\in[p-1,p/2)$ then $u$ vanishes identically after a finite time $T_e$ and
$$
\mathcal{P}(t) = \real^N \quad\text{ for }\quad t\in (0,T_e)\ .
$$
\end{proposition}
%%%%%%%%%%%%%%%%%%%%

Observe that Proposition~\ref{prop.nonloc} does not apply to $p=2$ as the range $[p-1,p/2)$ is empty in that case.

\begin{proof}
The occurrence of finite time extinction is a consequence of the discussion after Proposition~\ref{prop.ext}. Fix $t_0\in (0,T_e)$. By Proposition~\ref{prop.ext}, there exists $x_0\in\real^N$ such that $u(t_0,x_0)>0$. Let $B(x_0,\varrho)$ be the largest ball centered at $x_0$ and included in $\mathcal{P}(t_0)$.
We split the analysis into two cases:

\medskip

\noindent $\bullet \quad q\in (p-1,p/2)$: it follows from the
gradient estimates \cite[Theorem~1.3(iii)]{IL12} that for any $x\in B(x_0,\varrho)$, we have
$$
\left|\nabla u^{-(q-p+1)/(p-q)}(t_0,x)\right|\leq
K_0:=C\left(1+\|u_0\|_{\infty}^{(p-2q)/p(p-q)}t_0^{-1/p}\right),
$$
hence
$$
u^{-(q-p+1)/(p-q)}(t_0,x)\leq
u^{-(q-p+1)/(p-q)}(t_0,x_0)+K_0|x-x_0|,
$$
or equivalently
$$
\left[K_0|x-x_0|+u^{-(q-p+1)/(p-q)}(t_0,x_0)\right]^{-(p-q)/(q-p+1)}\leq
u(t_0,x).
$$
Assuming now for contradiction that $\varrho<\infty$, there exists $y\in\partial B(x_0,\varrho)$ such that $u(t_0,y)=0$, contradicting the previous lower bound. Consequently, $\varrho=\infty$ and thus $\mathcal{P}(t_0)=\real^N$.

\medskip

\noindent $\bullet \quad q=p-1$: it follows from \cite[Theorem~1.3(iv)]{IL12} that, for $x\in B(x_0,\varrho)$,
$$
|\nabla\log u(t_0,x)|\leq
K_0:=C\left(1+\|u_0\|_{\infty}^{(2-p)/p}t_0^{-1/p}\right),
$$
whence
$$
\log u(t_0,x_0)\leq\log u(t_0,x)+K_0|x-x_0|,
$$
which implies
$$
0<u(t_0,x_0)e^{-K_0|x-x_0|}\leq u(t_0,x),
$$
for any $x\in B(x_0,\varrho)$. Arguing as in the previous case, this readily implies that $\varrho=\infty$ and thus $\mathcal{P}(t_0)=\real^N$.
\end{proof}

%%%%%%%%%%%%%%%%%%%%
%%%%%%%%%%%%%%%%%%%%
\section{Upper and lower bounds at the extinction time}\label{sec.ulbet}
%%%%%%%%%%%%%%%%%%%%
%%%%%%%%%%%%%%%%%%%%

An interesting consequence of the localization property established in the previous section is the derivation of temporal upper and lower bounds for the $L^\infty$-norm of solutions to \eqref{eq1}-\eqref{eq2} with compactly supported initial data.

%%%%%%%%%%%%%%%%%%%%
\begin{proposition}\label{prop.ulbet}
Let $u_0$ be an initial condition satisfying \eqref{hypIC} and denote the corresponding solution to \eqref{eq1}-\eqref{eq2} by $u$. Assume further that the exponents $p$ and $q$ satisfy \eqref{exp} and that $u_0$ is compactly supported. Then $u$ vanishes identically at a finite time $T_e>0$ and there is $C_1>0$ such that
\begin{equation}
\|u(t)\|_\infty \ge C_1 (T_e-t)^{1/(1-q)}\ , \qquad t\in [0,T_e]\ . \label{zz7}
\end{equation}
In addition, given $\vartheta\in (0,1)$, there is $C_2(\vartheta)>0$ such that
\begin{equation}
\|u(t)\|_\infty \le C_2(\vartheta) (T_e-t)^{\vartheta(p-q)/(p-2q)}\ , \qquad t\in [T_e/2,T_e]\ . \label{zz8}
\end{equation}
\end{proposition}
%%%%%%%%%%%%%%%%%%%%

\begin{proof}
We first infer from the compactness of the support of $u_0$, Proposition~\ref{prop.ext}, and Proposition~\ref{prop.loc} that the extinction time $T_e$ of $u$ is positive and finite and that  there exists $R>0$ such that
\begin{equation}
\mathcal{P}(t) \subseteq B(0,R)\ , \qquad t\in [0,T_e]\ . \label{zz9}
\end{equation}
We next recall the following gradient estimates, the first one being
proved in \cite[Theorem~1.3(v)]{IL12} for $p\in (p_c,2)$ and in
Proposition~\ref{prop.gradest} for $p=2$:
\begin{equation}\label{grad.est1}
\left|\nabla u^{(p-q-1)/(p-q)}(t,x)\right|\leq
C_1\left(1+\|u(s)\|_{\infty}^{(p-2q)/p(p-q)}(t-s)^{-1/p}\right),
\end{equation}
and the next one proved in \cite[Theorem~1.7]{IL12} for $p\in (p_c,2)$ and in \cite[Theorem~2]{GGK03} for $p=2$:
\begin{equation}\label{grad.est2}
|\nabla u(t,x)|\leq C_2\|u(s)\|_{\infty}^{1/q}(t-s)^{-1/q}
\end{equation}
for $(t,x)\in (0,\infty)\times\real^N$ and $s\in [0,t)$, the constants $C_1$ and $C_2$ depending only on $N$, $p$, and $q$. Thanks to \eqref{zz9} and \eqref{grad.est2}, we may argue as in \cite[Proposition~3.5]{HJBook} and establish that\begin{equation*}
\|u(t)\|_{\infty} \geq C (T_e-t)^{1/(1-q)}, \qquad t\in
(0,T_e)\ .
\end{equation*}

We next deduce from \eqref{grad.est1} that
$$
\left|\nabla u^{(p-q-1)/(p-q)}(t,x)\right|\leq C\ , \qquad (t,x)\in [T_e/2,T_e] \times\real^N\ ,
$$
so that
\begin{equation}
|\nabla u(t,x)|\leq C u(t,x)^{1/(p-q)}, \qquad \
x\in \mathcal{P}(t)\ , \quad  t\in[T_e/2,T_e] \ . \label{zz10}
\end{equation}
Consider $\varrho\ge 1$ and $t\in [T_e/2,T_e]$. We infer from \eqref{eq1} that
\begin{align*}
\frac{1}{\varrho+1} \frac{d}{dt} \|u(t)\|_{\varrho+1}^{\varrho+1} & = - \int_{\real^N} \left( \varrho u^{\varrho-1} |\nabla u|^p + u^\varrho |\nabla u|^q \right)\ dx \\
& = - \int_{\mathcal{P}(t)} \left( \varrho u^{\varrho-1} |\nabla u|^p + u^\varrho |\nabla u|^q \right)\ dx\ .
\end{align*}
Integrating with respect to time over $(t,T_e)$ and using \eqref{zz10}, we obtain
\begin{align*}
\|u(t)\|_{\varrho+1}^{\varrho+1} & = (\varrho+1) \int_t^{T_e} \int_{\mathcal{P}(s)} \left( \varrho u^{\varrho-1} |\nabla u|^p + u^\varrho |\nabla u|^q \right)\ dxds \\ & \le C(1+\varrho)^2 \int_t^{T_e} \int_{\mathcal{P}(s)} u^{(\varrho(p-q)+q)/(p-q)}\ dxds\ .
\end{align*}
It then follows from \eqref{zz9}, the H\"older inequality, and the time monotonicity of $s\mapsto \|u(s)\|_{\varrho+1}$ that
\begin{align*}
\|u(t)\|_{\varrho+1}^{\varrho+1} & \le C (1+\varrho)^2 |B(0,R)|^{(p-2q)/(\varrho+1)(p-q)} \int_t^{T_e} \|u(s)\|_{\varrho+1}^{(\varrho(p-q)+q)/(p-q)} \ ds \\
& \le C(\varrho) (T_e-t) \|u(t)\|_{\varrho+1}^{(\varrho(p-q)+q)/(p-q)}\ .
\end{align*}
Consequently,
\begin{equation}
\|u(t)\|_{\varrho+1} \le C(\varrho) (T_e-t)^{(p-q)/(p-2q)}\ , \qquad t\in [T_e/2,T_e]\ . \label{zz11}
\end{equation}
We finally combine \eqref{zz10} and the Gagliardo-Nirenberg inequality to obtain, for $t\in [T_e/2,T_e]$,
\begin{align*}
\|u(t)\|_\infty & \le C(\varrho) \|\nabla u(t)\|_\infty^{N/(\varrho+1+N)} \|u(t)\|_{\varrho+1}^{(\varrho+1)/(\varrho+1+N)} \\
& \le C(\varrho) \|u(t)\|_\infty^{N/(\varrho+1+N)(p-q)} \|u(t)\|_{\varrho+1}^{(\varrho+1)/(\varrho+1+N)} \ .
\end{align*}
Consequently
$$
\|u(t)\|_\infty \le C(\varrho) \|u(t)\|_{\varrho+1}^{(\varrho+1)(p-q)/((\varrho+1)(p-q)+N(p-1-q))} \ ,
$$
which, together with \eqref{zz11}, gives \eqref{zz8} as $\varrho$ can be chosen arbitrarily large.
\end{proof}

When $p=2$ the upper bound \eqref{zz8} can be improved, allowing the value $\vartheta=1$ in \eqref{zz8}.

%%%%%%%%%%%%%%%%%%%%
\begin{proposition}\label{prop.impub}
Let $u_0$ be an initial condition satisfying \eqref{hypIC} and denote the corresponding solution to \eqref{eq1}-\eqref{eq2} by $u$. Assume further that $p=2$, $q \in (0,1)$, and that $u$ has a finite extinction time $T_e >0$. Then there is $C_3>0$ depending on $q$, $T_e$, and $u_0$ such that
\begin{equation*}
\|u(t)\|_\infty \le C_3 (T_e-t)^{(2-q)/(2-2q)}\ , \qquad t\in [0,T_e]\ .
\end{equation*}
\end{proposition}
%%%%%%%%%%%%%%%%%%%%

\begin{proof}
 We adapt the proof of \cite[Proposition~2.2]{HV92}. Assume for contradiction that, for each $n \ge 1$, there are
    $x_n \in \real^N$ and $t_n \in (0,T_e)$ such that
    \begin{equation}\label{pub2}
      u(t_n,x_n) \ge n (T_e - t_n)^{(2-q)/(2-2q)} \ .
    \end{equation}
    Since $u \le \| u_0 \|_\infty$ and $q<1$, it readily follows from \eqref{pub2} that
    \begin{equation}\label{pub3}
      \lim\limits_{n \to \infty} t_n = T_e \ .
    \end{equation}
    As $u(T_e) \equiv 0$, we infer from the variation of constants formula for \eqref{eq1} that for $t \in (0,T_e)$
    \begin{equation}\label{pub4}
      0 = u(T_e) = e^{(T_e-t) \Delta} u (t) - \int_t^{T_e} e^{(T_e -s) \Delta} |\nabla u(s)|^q ds \ .
    \end{equation}
    Now, for $t \in (T_e/2, T_e)$, it follows from Proposition~\ref{prop.gradest} and the properties of the fundamental solution of the heat equation
    that there is $c_1 >0$ such that
    \begin{equation}\label{pub5}
    \begin{split}
     I(t,x) \ & := \left( \int_t^{T_e} e^{(T_e -s) \Delta} |\nabla u(s)|^q ds \right) (x)\\
        &=  \int_t^{T_e} \frac{1}{(4\pi (T_e -s))^{N/2}} \int_{\real^N} \exp\left( - \frac{|x-y|^2}{4(T_e-s)}
        \right) |\nabla u(s,y)|^q dyds \\
        &=  \left( \frac{2-q}{1-q} \right)^q \int_t^{T_e} \frac{1}{(4\pi (T_e -s))^{N/2}}\\
        &\qquad \times \int_{\real^N}
        \exp\left( - \frac{|x-y|^2}{4(T_e-s)} \right) u^{q/(2-q)}(s,y) |\nabla u^{(1-q)/(2-q)}(s,y)|^q dyds\\
        &\le c_1 \int_t^{T_e} \frac{1}{(4\pi (T_e -s))^{N/2}} \int_{\real^N}
        \exp\left( - \frac{|x-y|^2}{4(T_e-s)} \right) u^{q/(2-q)}(s,y) dyds \\
        &\le c_1 \int_t^{T_e} \left[ \frac{1}{(4\pi (T_e -s))^{N/2}} \int_{\real^N}
        \exp\left( - \frac{|x-y|^2}{4(T_e-s)} \right) u(s,y) dy \right]^{q/(2-q)}
        ds,
    \end{split}
    \end{equation}
    where we use Jensen's inequality for concave functions in the last step in view of $q <2-q$. Introducing
    \begin{align*}
     &h(t,x) := \left( e^{(T_e-t)\Delta} u(t) \right)(x) = \frac{1}{(4\pi (T_e -t))^{N/2}} \int_{\real^N} \exp\left( - \frac{|x-y|^2}{4(T_e-t)} \right) u(t,y) dy \ ,\\
     &H(t,x) := \int_t^{T_e} h^{q/(2-q)} (s,x) ds
    \end{align*}
    for $(t,x) \in (T_e/2, T_e) \times \real^N$, we infer from \eqref{pub4} and \eqref{pub5} that
    \begin{equation}\label{pub6}
      h(t,x) = I(t,x) \le c_1 H(t,x), \qquad (t,x) \in (T_e/2, T_e) \times \real^N \ .
    \end{equation}
    This implies that
    $$
      - \partial_t H(t,x) = h^{q/(2-q)} (t,x) \le c_2 H^{q/(2-q)} (t,x)
    $$
    with some $c_2 >0$, whence an integration shows that
    $$
      - \frac{2-q}{2-2q}\left( H^{(2-2q)/(2-q)} (T_e,x) -  H^{(2-2q)/(2-q)} (t,x) \right) \le c_2 (T_e-t) \ ,
    $$
    for $(t,x) \in (T_e/2, T_e) \times \real^N$.
    Consequently, using $H(T_e) \equiv 0$ and \eqref{pub6}, we obtain $c_3>0$ such that
    \begin{equation}\label{pub7}
      h(t,x) \le c_1 H(t,x) \le c_3 (T_e-t)^{(2-q)/(2-2q)} \ , \qquad (t,x) \in (T_e/2, T_e) \times \real^N \ .
    \end{equation}
    Now in view of \eqref{pub3} there is $n_0 \in \mathbb{N}$ such that $t_n \in (T_e/2, T_e)$ for all $n \ge n_0$.
    We use once more Proposition~\ref{prop.gradest} along with \eqref{pub2} to obtain $c_4 >0$ such that, for all $n \ge n_0$ and
    $x \in \real^N$,
    \begin{align*}
      u^{(1-q)/(2-q)}(t_n,x) \ & \ge u^{(1-q)/(2-q)}(t_n,x_n) - c_4 |x-x_n|\\
        & \ge n^{(1-q)/(2-q)} (T_e-t_n)^{1/2} - c_4 |x-x_n| \ge \frac{n^{(1-q)/(2-q)}}{2} (T_e-t_n)^{1/2} \ ,
    \end{align*}
    provided that $x \in B\left( x_n, \frac{n^{(1-q)/(2-q)}}{2c_4} (T_e-t_n)^{1/2} \right)$. We then infer from the latter estimate, \eqref{pub7}, the definition of $h$, and the nonnegativity of $u$ that
    \begin{align*}
     c_3 (T_e-t_n)^{(2-q)/(2-2q)} \
        & \ge h(t_n,x_n)\\
        & \ge \frac{1}{(4\pi (T_e -t_n))^{N/2}} \int_{B\left( x_n, \frac{n^{(1-q)/(2-q)}}{2c_4} (T_e-t_n)^{1/2} \right)}
      \exp\left( - \frac{|x_n-y|^2}{4(T_e-t_n)} \right)\\
        & \qquad \times \frac{n}{2^{(2-q)/(1-q)}} (T_e -t_n)^{(2-q)/(2-2q)} dy
    \end{align*}
    for all $n \ge n_0$. We conclude that
    $$
      c_3 \ge \frac{n}{ 2^{(2-q)/(1-q)} (4\pi)^{N/2}} \int_{B\left( 0, \frac{n^{(1-q)/(2-q)}}{2c_4} \right)} \exp \left(-\frac{|z|^2}{4} \right) dz
    $$
    for all $n \ge n_0$, and a contradiction.
\end{proof}

%%%%%%%%%%%%%%%%%%%%
%%%%%%%%%%%%%%%%%%%%
\section{Non-extinction}\label{sec.nonext}
%%%%%%%%%%%%%%%%%%%%
%%%%%%%%%%%%%%%%%%%%

This section is devoted to the proof of Theorem~\ref{th.noninst}.
As in \cite{BLSS02} (which only deals with the case $p=2$), it
relies on the comparison principle, this time with suitable
subsolutions which we construct now.

%%%%%%%%%%%%%%%%%%%%
\begin{lemma}\label{lem.sub}
Assume that $(p,q)$ satisfies \eqref{exp}. There exists $b_0
>0$ depending only on $p$ and $q$ such that, given $T>0$, $a>0$, and
$b\in (0,b_0)$, the function
\begin{equation}\label{subform}
w(t,x):=(T-t)^{1/(1-q)}\left(a+b|x|^{\theta}\right)^{-\gamma},
\qquad \theta=\frac{p}{p-1}, \ \gamma=\frac{q(p-1)}{p(1-q)},
\end{equation}
is a subsolution to \eqref{eq1} in $(0,T)\times\real^N$ provided $a$ is large enough.
\end{lemma}
%%%%%%%%%%%%%%%%%%%%

\begin{proof}
Since $w$ is radially symmetric and letting $r=|x|$ as usual, we
note that
$$
\mathcal{L}w=\partial_tw-|\partial_rw|^{p-2}\left[(p-1)\partial_{r}^2 w+\frac{N-1}{r}\partial_rw\right]+|\partial_rw|^q.
$$
Setting $y:=a+b|x|^{\theta}$, we easily have:
$$
\partial_tw(t,x)=-\frac{1}{1-q}(T-t)^{q/(1-q)}y^{-\gamma}, \qquad
\partial_rw(t,x)=-\gamma\theta b(T-t)^{1/(1-q)}r^{\theta-1}y^{-\gamma-1},
$$
and
$$
\partial_{r}^2w(t,x)=-\gamma\theta
b(T-t)^{1/(1-q)}y^{-\gamma-2}r^{\theta-2}\left[(\theta-1)y-(1+\gamma)\theta
br^{\theta}\right],
$$
whence, after easy manipulations,
\begin{equation*}
\begin{split}
\mathcal{L}w&=(T-t)^{q/(1-q)}y^{-\gamma}\left[(\gamma\theta
b)^{q}r^{q(\theta-1)}y^{\gamma(1-q)-q}-\frac{1}{1-q}\right]\\&+(\gamma\theta
b)^{p-1}(T-t)^{(p-1)/(1-q)}r^{\theta(p-1)-p}y^{-(p-1)\gamma-p}\\&\times\left[((p-1)(\theta-1)+N-1)y-(1+\gamma)\theta
b(p-1)r^{\theta}\right],
\end{split}
\end{equation*}
or equivalently,
\begin{equation*}
\begin{split}
(T-t)^{-q/(1-q)}y^{\gamma}\mathcal{L}w&=(\gamma\theta
b)^{q}r^{q(\theta-1)}y^{\gamma(1-q)-q}-\frac{1}{1-q}\\&+(\gamma\theta
b)^{p-1}(T-t)^{(p-1-q)/(1-q)}r^{\theta(p-1)-p}y^{(2-p)\gamma-p}\\&\times\left[((p-1)(\theta-1)+N-1-(p-1)\theta(1+\gamma))y+(1+\gamma)\theta(p-1)a\right]\\
&=(\gamma\theta
b)^{q}r^{q(\theta-1)}y^{\gamma(1-q)-q}-\frac{1}{1-q}\\&+(\gamma\theta
b)^{p-1}(T-t)^{(p-1-q)/(1-q)}r^{\theta(p-1)-p}y^{(2-p)\gamma-p}\\&\times\left[(1+\gamma)\theta(p-1)a+(N-1-(p-1)(1+\theta\gamma))y\right].
\end{split}
\end{equation*}
Recalling now the values of $\theta$ and $\gamma$ from
\eqref{subform} and that $q<p-1$, we obtain
\begin{equation}\label{interm10}
\begin{split}
(T-t)^{-q/(1-q)}y^{\gamma}\mathcal{L}w&=(\gamma\theta)^qb^{q/\theta}(br^{\theta})^{q/p}y^{\gamma(1-q)-q}-\frac{1}{1-q}\\&+(\gamma\theta
b)^{p-1}(T-t)^{(p-1-q)/(1-q)}y^{(2-p)\gamma-p}\\&\times\left[(1+\gamma)pa+\left(N-1-\frac{p-1}{1-q}\right)y\right]\\
&\leq(\gamma\theta)^qb^{q/\theta}-\frac{1}{1-q}\\&+(\gamma\theta
b)^{p-1}T^{(p-1-q)/(1-q)}y^{(2-p)\gamma-p+1}[(1+\gamma)p+N-1]\\&\leq(\gamma\theta)^qb^{q/\theta}-\frac{1}{1-q}\\&+(\gamma\theta
b)^{p-1}T^{(p-1-q)/(1-q)}a^{(2-p)\gamma-p+1}[(1+\gamma)p+N-1],
\end{split}
\end{equation}
where we have used in the last inequality the fact that
\begin{equation}
(2-p)\gamma-p+1=\frac{(p-1)(2q-p)}{p(1-q)}<0 \label{interm101}
\end{equation}
due to $q<p-1\le p/2$. Setting
$$
b_0 := \left( 2(1-q) (\gamma\theta)^q \right)^{-\theta/q}\ ,
$$
it follows from \eqref{interm10} that, for $b\in (0,b_0)$,
\begin{align*}
(T-t)^{-q/(1-q)}y^{\gamma}\mathcal{L}w & \le (\gamma\theta
b)^{p-1}T^{(p-1-q)/(1-q)}a^{(2-p)\gamma-p+1}[(1+\gamma)p+N-1] \\
& \qquad - \frac{1}{2(1-q)}.
\end{align*}
We use once more \eqref{interm101} to conclude that the right hand side of the previous inequality is negative for $a$ large enough, that is, $w$ is a subsolution to
\eqref{eq1} in $(0,T)\times\real^N$.
\end{proof}

With this technical lemma, we are in a position to complete the proof of Theorem~\ref{th.noninst}.

\begin{proof}[Proof of Theorem~\ref{th.noninst}]
Let $u_0$ be an initial condition satisfying \eqref{hypIC} and
\eqref{init.decay2}. Fix $b\in (0,b_0)$ and $T>0$. Since
$$
\lim\limits_{|x|\to\infty}u_0(x)|x|^{q/(1-q)}=\infty,
$$
there exists $R>0$ sufficiently large such that
\begin{equation}\label{compout}
\begin{split}
u_0(x) & \geq T^{1/(1-q)} b^{-q(p-1)/p(1-q)} |x|^{-q/(1-q)} \\
& \geq T^{1/(1-q)}\left(a+b|x|^{p/(p-1)}\right)^{-q(p-1)/p(1-q)},
\end{split}
\end{equation}
for any $a>0$ and $x\in\real^N\setminus B(0,R)$. It remains to show
that it is possible to choose $a>0$ such that \eqref{compout} also
holds true inside $B(0,R)$. Since $u_0>0$ in $\real^N$ and
$\overline{B(0,R)}$ is a compact set, we infer from the continuity
of $u_0$ that
$$
\delta:=\inf\{u_0(x): x\in B(0,R)\}>0.
$$
Choose then $a>0$ sufficiently large such that
$$
T^{1/(1-q)}a^{-q(p-1)/p(1-q)}<\delta\ ,
$$
and set
$$
w(t,x):=(T-t)^{1/(1-q)}\left(a+b|x|^{p/(p-1)}\right)^{-q(p-1)/p(1-q)}, \qquad (t,x)\in (0,T)\times\real^N.
$$
Then
\begin{equation}\label{compin}
u_0(x)\geq\delta>T^{1/(1-q)}a^{-q(p-1)/p(1-q)}\geq w(0,x),
\end{equation}
for any $x\in B(0,R)$ and we combine \eqref{compout} and
\eqref{compin} to conclude that $u_0(x)\geq w(0,x)$ for any
$x\in\real^N$. Owing to the choice $b\in (0,b_0)$, the function $w$
is a subsolution to \eqref{eq1} according to Lemma~\ref{lem.sub},
and the comparison principle ensures that
$$
u(t,x)\geq w(t,x)>0, \qquad {\rm for \ any} \ t\in(0,T), \
x\in\real^N.
$$
Consequently, $\mathcal{P}(t)=\real^N$ for all $t\in (0,T)$ and, since $T>0$ is arbitrarily chosen, we conclude that
$\mathcal{P}(t)=\real^N$ for any $t>0$, as stated.
\end{proof}

%%%%%%%%%%%%%%%%%%%%
%%%%%%%%%%%%%%%%%%%%
\section{Improved finite time extinction property}\label{sec.improv}
%%%%%%%%%%%%%%%%%%%%
%%%%%%%%%%%%%%%%%%%%

This section is devoted to the proof of Theorem~\ref{th.improv},
which improves the range of initial data for which finite time
extinction takes place. As in the previous section, the main argument
is again the comparison principle, this time with suitable
supersolutions. The idea to construct them is adapted from \cite[Lemma~7]{BLSS02}, where similar supersolutions are built in the semilinear case $p=2$ and $q\in(0,1)$. To this end we define
\begin{equation}\label{selfsimexp}
\a:=\frac{p-q}{p-2q}>0, \qquad \b:=\frac{q-p+1}{p-2q}<0,
\end{equation}
which are the usual self-similar exponents associated to Eq.
\eqref{eq1} and look for self-similar supersolutions of \eqref{eq1}.

%%%%%%%%%%%%%%%%%%%%
\begin{lemma}\label{lem.super}
Assume that $(p,q)$ satisfies \eqref{exp} and that $p<2$. There
exists $A_0>0$ depending only on $N$, $p$, and $q$ such that, for
all $T>0$ and $A\in (0,A_0)$, the function
\begin{equation}\label{superform}
W(t,x):=(T-t)^{\a}f(|x|(T-t)^{\b}), \qquad
(t,x)\in(0,T)\times\real^N,
\end{equation}
with
$$
f(y):= A(1+y^2)^{-\gamma}, \quad y\in\real, \qquad \gamma:=\frac{q}{2(1-q)},
$$
is a supersolution to \eqref{eq1} in $(0,T)\times\real^N$.
\end{lemma}
%%%%%%%%%%%%%%%%%%%%

\begin{proof}
Since $W$ is radially symmetric, we again have (with $r=|x|$)
$$
\mathcal{L}W=\partial_t W-(p-1)|\partial_rW|^{p-2}\partial_{r}^2W-\frac{N-1}{r}|\partial_rW|^{p-2}\partial_rW+|\partial_rW|^q.
$$
Setting $y:=|x|(T-t)^{\b}=r(T-t)^{\b}$ and taking into account that
$$
\a-1=(p-1)(\a+\b)+\b=q(\a+\b)=\frac{q}{p-2q}>0,
$$
we readily find
\begin{equation*}
\begin{split}
\mathcal{L}W=(T-t)^{\a-1}&\left[-\a f(y)-\b
y\partial_{y}f(y)-(p-1)|\partial_{y}f(y)|^{p-2} \partial^2_{y} f(y)\right.\\
&\left. \qquad -\frac{N-1}{y}|\partial_{y}f(y)|^{p-2}\partial_{y}f(y)+|\partial_{y}f(y)|^{q}\right].
\end{split}
\end{equation*}
Since
$$
\partial_yf(y)=-2A\gamma y(1+y^2)^{-\gamma-1}
$$
and
$$
\partial^2_{y}f(y)=-2A\gamma(1+y^2)^{-\gamma-1}\left[1-\frac{2(\gamma+1)y^2}{1+y^2}\right],
$$
we further obtain
\begin{equation*}
\begin{split}
\mathcal{L}W&=(T-t)^{\a-1}(1+y^2)^{-(\gamma+1)}\left[-\a
A-(\a-2\b\gamma)Ay^2\right.\\
& \qquad \left.+(p-1)(2A\gamma)^{p-1}\frac{(1+y^2)^{(2-p)(\gamma+1)}}{y^{2-p}}\left(1-\frac{2(\gamma+1)y^2}{1+y^2}\right)\right.\\
&\qquad \left.+(N-1)(2A\gamma)^{p-1}\frac{(1+y^2)^{(2-p)(\gamma+1)}}{y^{2-p}}+(2A\gamma)^qy^q(1+y^2)^{(1-q)(\gamma+1)}\right]\\
&=(T-t)^{\a-1}(1+y^2)^{-(\gamma+1)}H(y),
\end{split}
\end{equation*}
where
\begin{equation*}
\begin{split}
H(y)&:=(2A\gamma)^{p-1}\frac{(1+y^2)^{(2-p)(\gamma+1)}}{y^{2-p}}\left[p+N-2-\frac{2(p-1)(\gamma+1)y^2}{1+y^2}\right]\\
&+(2A\gamma)^qy^q(1+y^2)^{(1-q)(\gamma+1)}-A\a-(\a-2\b\gamma)Ay^2.
\end{split}
\end{equation*}
We first note that
\begin{equation}\label{interm12}
\begin{split}
\frac{1}{2}(2A\gamma)^q&y^q(1+y^2)^{(1-q)(\gamma+1)}-(\a-2\b\gamma)Ay^2\\
&\geq\frac{(2A\gamma)^q}{2}y^q y^{2(1-q)(\gamma+1)} -(\a-2\b\gamma)Ay^2\\
&\geq A y^2\left[\frac{(2\gamma)^q}{2}A^{q-1}-(\a-2\b\gamma)\right].
\end{split}
\end{equation}
In order to estimate the remaining terms in the expression of
$H(y)$, we have to split the analysis into two cases according to
the values of $y$. More precisely, set
$$
y_0:=\frac{1}{\sqrt{4(\gamma+1)}}\ .
$$

\medskip

\noindent $\bullet$ If $y\in[0,y_0]$, we find
\begin{equation}\label{interm13}
\begin{split}
(2A\gamma)^{p-1}&\frac{(1+y^2)^{(2-p)(\gamma+1)}}{y^{2-p}}\left[p+N-2-\frac{2(p-1)(\gamma+1)y^2}{1+y^2}\right]-A\a\\
&\geq(2A\gamma)^{p-1}  y^{p-2} \left[p-1-\frac{2(p-1)(\gamma+1)y^2}{1+y^2}\right]-A\a\\
&\geq (2A\gamma)^{p-1} (p-1) y^{p-2} \left[1-2(\gamma+1)y_0^2 \right]-A\a\\
&\geq (2A\gamma)^{p-1}\frac{p-1}{2} y_0^{p-2}-A\a \\
&\geq
A\left[\frac{(p-1)(2\gamma)^{p-1}}{2y_0^{2-p}}A^{p-2}-\a\right].
\end{split}
\end{equation}

\medskip

\noindent $\bullet$ If $y\geq y_0$, we have
\begin{equation}\label{interm14}
\begin{split}
& (2A\gamma)^{p-1}\frac{(1+y^2)^{(2-p)(\gamma+1)}}{y^{2-p}}\left[p+N-2-\frac{2(p-1)(\gamma+1)y^2}{1+y^2}\right]\\
& \qquad\qquad \qquad +\frac{1}{2}(2A\gamma)^qy^q(1+y^2)^{(1-q)(\gamma+1)}-A\a\\
&\geq -2(p-1)(1+\gamma)(2A\gamma)^{p-1}\frac{(1+y^2)^{(2-p)(\gamma+1)}}{y^{2-p}} \frac{y^2}{1+y^2} \\
& \qquad\qquad \qquad + \frac{(2A\gamma)^q}{2} y^q y^{2(1-q)(\gamma+1)}-A\a\\
&\geq -2(p-1)(1+\gamma)(2A\gamma)^{p-1}\frac{(1+y^2)^{(2-p)(\gamma+1)}}{y^{2-p}} + \frac{(2A\gamma)^q}{2} y^2 - A\a\\
&\geq\frac{(2A\gamma)^q}{4} y_0^2-A\a +\frac{(2A\gamma)^q}{4} y^2\\
& \qquad\qquad \qquad -2(p-1)(\gamma+1) (2A\gamma)^{p-1} \left( 1 + \frac{1}{y_0^2} \right)^{(2-p)(\gamma+1)} y^{(2-p)(2(\gamma+1)-1)}\\
&\geq
A \left[\frac{(2\gamma)^q}{4}y_0^2 A^{q-1}-\alpha \right] + (2A\gamma)^{p-1} y^{(2-p)/(1-q)}\\
& \qquad \times\left[\frac{(2\gamma)^{q-p+1}}{4} y_0^{(p-2q)/(1-q)}
A^{q-p+1} -2 (p-1)(\gamma+1) \left( \frac{1+y_0^2}{y_0^2}
\right)^{(2-p)(\gamma+1)}\right].
\end{split}
\end{equation}
Since $0<q<p-1<1$ and $p<2$, we realize that the right-hand sides of \eqref{interm12}, \eqref{interm13}, and \eqref{interm14} are simultaneously positive provided $A$ is sufficiently small. It then follows that $W$ defined in \eqref{superform} is a supersolution to \eqref{eq1} in $(0,T)\times\real^N$ for $A$ small enough.
\end{proof}

Completing the proof of Theorem~\ref{th.improv} becomes now an easy task. A noteworthy fact to be mentioned is that the choice of $A>0$ in Lemma~\ref{lem.super} is independent of the fixed extinction time $T>0$.

\begin{proof}[Proof of Theorem~\ref{th.improv}]
For $p=2$ and $q\in (0,1)$ Theorem~\ref{th.improv} is proved in
\cite[Theorem~2]{BLSS02}. Consider now $p<2$. Let $u_0$ be an
initial condition satisfying \eqref{hypIC} and \eqref{init.decay3}
and $T>1$. We fix $A\in (0,A_0)$ so that the function
$$
W(t,x)=(T-t)^{\a} A \left(1+(T-t)^{2\b}|x|^2\right)^{-q/2(1-q)}, \qquad (t,x)\in (0,T)\times\real^N,
$$
is a supersolution to \eqref{eq1} in $(0,T)\times\real^N$ according to Lemma~\ref{lem.super}. We notice that, for any
$x\in\real^N$,
\begin{equation*}
\begin{split}
W(0,x)&= T^{\a} A \left(1+T^{2\b}|x|^2\right)^{-q/2(1-q)}\\
&\geq
T^{\a} A \left(T^{2\b}+T^{2\b}|x|^2\right)^{-q/2(1-q)}\\
&=T^{\a-\b q/(1-q)} A \left(1+|x|^2\right)^{-q/2(1-q)}\\
&=T^{1/(1-q)} A \left(1+|x|^2\right)^{-q/2(1-q)}.
\end{split}
\end{equation*}
Since $1+|x|^2\leq(1+|x|)^2$, for any $x\in\real^N$, we further
obtain that
\begin{equation*}
\begin{split}
W(0,x)&\geq T^{1/(1-q)} A (1+|x|)^{-q/(1-q)}\\
&=\frac{T^{1/(1-q)} A }{C_0}C_0(1+|x|)^{-q/(1-q)}\geq
\frac{T^{1/(1-q)} A }{C_0}u_0(x),
\end{split}
\end{equation*}
for any $x\in\real^N$. Letting $T>0$ to be sufficiently large such
that
$$
\frac{T^{1/(1-q)} A }{C_0}>1,
$$
we have $W(0,x)\geq u_0(x)$ for any $x\in\real^N$. Consequently, by
the comparison principle,
$$
W(t,x)\geq u(t,x), \qquad (t,x)\in(0,T)\times\real^N\ .
$$
Noting that $W(t,0) = A(T-t)^\alpha$ and $W(t,x)\le A
(T-t)^{1/(1-q)} |x|^{-q/(1-q)}$ for $x\ne 0$, we realize that
$W(t,x)\longrightarrow 0$ as $t\to T$ for all $x\in\real^N$, which
in particular implies that $u$ vanishes in finite time. Moreover,
its extinction time satisfies $T_e\leq T$.
\end{proof}

%%%%%%%%%%%%%%%%%%%%
%%%%%%%%%%%%%%%%%%%%
\section{Single point extinction}\label{sec.single}
%%%%%%%%%%%%%%%%%%%%
%%%%%%%%%%%%%%%%%%%%

This rather long section is devoted to the proof of Theorem~\ref{th.single}. The single point extinction is an immediate
consequence of the double inclusion \eqref{incl.pos}, which is the main result to prove. This is divided into several steps, corresponding to subsections in the sequel.

%%%%%%%%%%%%%%%%%%%%
%%%%%%%%%%%%%%%%%%%%
\subsection{Lower bound for the positivity set}
%%%%%%%%%%%%%%%%%%%%
%%%%%%%%%%%%%%%%%%%%

We begin with the first inclusion in \eqref{incl.pos}.

%%%%%%%%%%%%%%%%%%%%
\begin{lemma}\label{lem.lower}
Let $u$ be a solution to the Cauchy problem \eqref{eq1}-\eqref{eq2} with an initial condition $u_0$ satisfying \eqref{hypIC} which is radially symmetric and non-increasing. Assume further that $u_0$ is compactly supported and denote its finite extinction time by $T_e$. There exists $\varrho_1>0$ such that
$$
B(0,\varrho_1(T_e-t)^{\sigma})\subseteq\mathcal{P}(t) \qquad {\rm for \ any } \ t\in(T_e/2,T_e) \quad\text{ with }\;\; \sigma:=\frac{p-q-1}{(p-q)(1-q)},
$$
where $\mathcal{P}(t)$ denotes the positivity set of $u$
at time $t$, see \eqref{pos.set}.
\end{lemma}
%%%%%%%%%%%%%%%%%%%%

\begin{proof}
Since $\|u(t)\|_\infty=u(t,0)$ due to the radial symmetry and monotonicity of $u(t)$ provided by Lemma~\ref{lem.radial}, we infer from Proposition~\ref{prop.ulbet} that
\begin{equation*}
u(t,0) \geq C_1(T_e-t)^{1/(1-q)}, \qquad t\in
(0,T_e).
\end{equation*}
Next, it readily follows from the mean-value theorem
and the gradient estimate \eqref{grad.est1} that
$$
u^{(p-q-1)/(p-q)}(t,x)-u^{(p-q-1)/(p-q)}(t,0)\geq -C|x|, \qquad
t\in[T_e/2,T_e],
$$
whence, using the above lower bound for $u(t,0)$,
$$
u^{(p-q-1)/(p-q)}(t,x)\ge C_1^{(p-q-1)/(p-q)}(T_e-t)^{\sigma}-C|x|>0,
$$
provided
$$
|x| <\frac{C_1^{\sigma(1-q)}}{C}(T_e-t)^{\sigma}, \quad
t\in(T_e/2,T_e),
$$
ending the proof.
\end{proof}

As already pointed out in the Introduction, Lemma~\ref{lem.lower}
holds true without requiring the assumptions \eqref{locvan} and
\eqref{init.data4} on $u_0$. It is in the next subsection where they
will be needed.

%%%%%%%%%%%%%%%%%%%%
%%%%%%%%%%%%%%%%%%%%
\subsection{Upper bound for the positivity set}
%%%%%%%%%%%%%%%%%%%%
%%%%%%%%%%%%%%%%%%%%

We move now to the more involved part, that is, to prove the second inclusion in \eqref{incl.pos} which turns out to be rather technical. Let $u_0$ be a radially symmetric and non-increasing initial condition satisfying \eqref{hypIC}. Assume further that $u_0$ is compactly supported in $B(0,R_0)$ for some $R_0>0$ and satisfies \eqref{locvan} for all $x_0\in\partial B(0,R_0)$ as well as \eqref{init.data4}. Denoting the corresponding solution to the Cauchy problem \eqref{eq1}-\eqref{eq2} by $u$ we infer from Theorem~\ref{th.inst} and Proposition~\ref{prop.iwt} that there is $T_e>0$ such that
\begin{equation}
\mathcal{P}(t) \subseteq B(0,R_0), \quad t\in (0,T_e), \quad\text{ and }\;\; \mathcal{P}(t)=\emptyset, \quad t\ge T_e. \label{zz1}
\end{equation}
Furthermore, Lemma~\ref{lem.radial} and the assumptions on $u_0$ guarantee that
\begin{equation}
\partial_r u(t,r) \le 0\ , \qquad (t,r) \in (0,T_e)\times (0,R_0)\ . \label{zz2}
\end{equation}

%%%%%%%%%%%%%%%%%%%%
\begin{lemma}\label{lem.J}
There exists $\delta>0$ such that
\begin{equation}\label{jay}
|\nabla u(t,x)| \geq \delta^{1/(p-1)} |x|^{1/(p-1-q)}
u(t,x)^{1/(p-q)}, \qquad (t,x)\in(0,T_e)\times B(0,R_0)\ .
\end{equation}
\end{lemma}
%%%%%%%%%%%%%%%%%%%%

\begin{proof}
We adapt an idea from \cite[Lemma~2.2]{FH} which takes its origin in
the study of blow-up problems, see \cite[Lemma~2.2]{FML}. For $z>
0$, we set
$$
a(z) := z^{(p-2)/2}\ , \qquad b(z) := z^{q/2}\ ,
$$
so that, using the radial variable $r:=|x|$, Eq.~\eqref{eq1} reads
\begin{equation}
\partial_t(r^{N-1}u) =\partial_r \left( r^{N-1}a(|\partial_r u|^2)\partial_r u \right) - r^{N-1} b(|\partial_ru|^2)\ , \qquad (t,r)\in (0,T_e)\times (0,R_0)\ . \label{Q2}
\end{equation}
We define the auxiliary function
$$
J(t,r):=r^{N-1} a(|\partial_ru(t,r)|^2)\partial_ru(t,r)+c(r)F(u(t,r)), \qquad (t,r)\in (0,T_e)\times (0,R_0),
$$
where the functions $c\geq0$ and $F\geq 0$ are to be determined later and assumed to satisfy
\begin{equation}
c(0)=0 \ , \quad F(0) = 0\ , \quad F'\ge 0\ , \quad F'' \le
0\ . \label{Q0}
\end{equation}
We aim at finding $c$ and $F$ such that $J\leq 0$ in $(0,T_e)\times
(0,R_0)$. Since $F(0)=0$, we first note that \eqref{zz1} and
\eqref{zz2} imply that
\begin{equation}
J(t,0)=0\ , \qquad  J(t,R_0)=R_0^{N-1}|\partial_r
u(t,R_0)|^{p-2}\partial_ru(t,R_0)\leq 0, \qquad t\in (0,T_e)\ ,
\label{zz14}
\end{equation}
as well as
\begin{equation*}
J(t,r) = - r^{N-1} |\partial_r u(t,r)|^{p-1} + c(r) F(u(t,r)), \qquad (t,r)\in (0,T_e)\times (0,R_0)\ .
\end{equation*}
The following calculations are performed at a formal level. A
rigorous justification requires some approximating procedures and
will be completed in the Appendix, a detailed account of the formal
calculations being given for the easiness of the reading.
Introducing
$$
g:=-\partial_r u \ge 0 \;\;\text{ and }\;\; a_1(z) := 2 z a'(z) + a(z)\ , \quad z>0\ ,
$$
we infer from \eqref{Q2} that
\begin{align*}
\partial_t J & = a_1(g^2) \partial_t \left( r^{N-1} \partial_r u \right) + c(r) F'(u) \partial_t u \\
& = a_1(g^2) \partial_t \left[ \partial_r \left( r^{N-1} u \right) - (N-1) r^{N-2} u \right] + \frac{c(r)}{r^{N-1}} F'(u) \partial_t (r^{N-1} u ) \\
& = a_1(g^2) \partial_r \left[ \partial_r \left( r^{N-1} a(g^2)\partial_r u \right) - r^{N-1} b(g^2)\right] \\
& \qquad + \left( \frac{c(r)}{r^{N-1}} F'(u) - \frac{N-1}{r} a_1(g^2) \right) \left[ \partial_r \left( r^{N-1}a(g^2)\partial_r u \right) - r^{N-1} b(g^2) \right] \\
& = a_1(g^2) \partial_r^2  \left( r^{N-1} a(g^2)\partial_r u \right) - 2 r^{N-1} (a_1b')(g^2) \partial_r u \partial_r^2 u\\
& \qquad + \left( \frac{c(r)}{r^{N-1}} F'(u) - \frac{N-1}{r} a_1(g^2) \right) \partial_r \left( r^{N-1}a(g^2)\partial_r u \right) - c(r) F'(u) b(g^2)\ .
\end{align*}
Next
\begin{align*}
\partial_r J & = \partial_r  \left( r^{N-1} a(g^2)\partial_r u \right) + c'(r) F(u) + c(r) F'(u) \partial_r u\ , \\
\partial_r^2 J & = \partial_r^2  \left( r^{N-1} a(g^2)\partial_r u \right) + c''(r) F(u) + 2 c'(r) F'(u) \partial_r u \\
& \qquad + c(r) F''(u) g^2 + c(r) F'(u) \partial_r^2 u\ ,
\end{align*}
and it follows from the formulas for $\partial_t J$ and $\partial_r^2 J$ that
\begin{align*}
\partial_t J - a_1(g^2) \partial_r^2 J & = 2 r^{N-1} (a_1b')(g^2)g \partial_r^2 u \\
& \qquad + \left( \frac{c(r)}{r^{N-1}} F'(u) - \frac{N-1}{r} a_1(g^2) \right) \partial_r \left( r^{N-1}a(g^2)\partial_r u \right) \\
& \qquad - c(r) F'(u) b(g^2) - a_1(g^2) c''(r) F(u) + 2 a_1(g^2) c'(r) F'(u) g \\
& \qquad - c(r) F''(u) a_1(g^2) g^2 - a_1(g^2) c(r) F'(u) \partial_r^2 u \ .
\end{align*}
We now use the formula for $\partial_r J$ to replace the terms involving $\partial_r^2 u$ in the above identity and obtain
\begin{align*}
\partial_t J - a_1(g^2) \partial_r^2 J & = \left( 2 b'(g^2)g - \frac{c(r)}{r^{N-1}} F'(u) \right) r^{N-1} \partial_r \left( a(g^2) \partial_r u \right) \\
& \qquad + \left( \frac{c(r)}{r^{N-1}} F'(u) - \frac{N-1}{r} a_1(g^2) \right) \partial_r \left( r^{N-1}a(g^2)\partial_r u \right) \\
& \qquad - c(r) F'(u) b(g^2) - a_1(g^2) c''(r) F(u) + 2 a_1(g^2) c'(r) F'(u) g \\
& \qquad - c(r) F''(u) a_1(g^2) g^2 \\
& = 2 b'(g^2)g \left[ \partial_r J - c'(r) F(u) + c(r) F'(u) g \right] + 2(N-1) r^{N-2} (ab')(g^2) g^2 \\
& \qquad - \frac{N-1}{r} c(r) F'(u) a(g^2) g - \frac{N-1}{r} a_1(g^2) \left[ \partial_r J - c'(r) F(u) + c(r) F'(u) g \right] \\
& \qquad - c(r) F'(u) b(g^2) - c''(r) F(u) a_1(g^2) + 2 c'(r) F'(u) a_1(g^2) g \\
& \qquad - c(r) F''(u) a_1(g^2) g^2 \ .
\end{align*}
Consequently
\begin{align*}
& \partial_t J - a_1(g^2) \partial_r^2 J + \left( \frac{N-1}{r} a_1(g^2) - 2 b'(g^2)g \right) \partial_ r J \\
& = 2 \left[ (N-1) r^{N-2} a(g^2) g - c'(r) F(u) \right] b'(g^2) g + \left[ \frac{N-1}{r} c'(r) - c''(r) \right] F(u) a_1(g^2) \\
& \qquad + \left[ 2 c'(r) a_1(g^2) - \frac{N-1}{r} c(r) (a+a_1)(g^2) \right] F'(u) g - c(r) F''(u) a_1(g^2) g^2 \\
& \qquad - \left[ b(g^2) - 2 b'(g^2) g^2 \right] c(r) F'(u)\ .
\end{align*}
Since
$$
a_1(z) = (p-1) a(z) \le a(z) \;\;\text{ and }\;\; b(z) - 2 z b'(z)
= (1-q) b(z)\ , \qquad z>0\ ,
$$
we infer from the non-negativity of $c$ and $g$, and the
monotonicity \eqref{Q0} of $F$, that
\begin{equation}
\partial_t J - a_1(g^2) \partial_r^2 J + \left( \frac{N-1}{r} a_1(g^2) - 2 b'(g^2)g \right) \partial_ r J \le \sum_{i=1}^3 \mathcal{R}_i\ , \label{Q8}
\end{equation}
where
\begin{align*}
\mathcal{R}_1 & := 2 \left[ (N-1) r^{N-2} a(g^2) g - c'(r) F(u) \right] b'(g^2) g\ , \\
\mathcal{R}_2 & := \left[ \frac{N-1}{r} c'(r) - c''(r) \right] F(u) a_1(g^2) \ , \\
\mathcal{R}_3 & := 2 \left[ c'(r) - \frac{N-1}{r} c(r) \right] F'(u) a_1(g^2) g - c(r) F''(u) a_1(g^2) g^2 \\
& \qquad\qquad - (1-q) c(r) F'(u) b(g^2)\ .
\end{align*}
We now choose
$$
c(r) = r^\lambda\ , \quad r\ge 0\ , \qquad F(z) = \delta z^\beta\ , \quad z\ge 0\ ,
$$
where $\lambda > N$, $\delta\in (0,1)$, and $\beta\in (0,1)$ are yet to be determined. Note that the latter constraint on $\beta$ complies with \eqref{Q0}. With this choice,
\begin{equation}
\mathcal{R}_2 = - \lambda (\lambda-N) \delta (p-1)  r^{\lambda-2}  u^\beta g^{p-2} \ ,  \label{Q9}
\end{equation}
while
\begin{equation}
\mathcal{R}_1 = 2 r^{\lambda-1} b'(g^2) g \left[ (N-1) r^{N-1-\lambda} a(g^2) g - \delta\lambda u^\beta \right]\ . \label{Q10}
\end{equation}
Moreover, since $\lambda> N$ and $\beta \in (0,1)$,
\begin{align}
\mathcal{R}_3 & = \delta \left[ 2 (\lambda-N+1) \beta (p-1) r^{\lambda-1} u^{\beta-1} g^{p-1} \right. \nonumber \\
& \qquad\qquad \left. + \beta (1-\beta) (p-1) r^\lambda u^{\beta-2} g^p - (1-q) \beta r^\lambda u^{\beta-1} g^q \right] \nonumber\\
& \le \beta \delta \left[ \frac{2(\lambda-N+1)}{r} g^{p-1-q} + \frac{1-\beta}{u} g^{p-q} - (1-q) \right] r^\lambda u^{\beta-1} g^q\ . \label{Q11}
\end{align}
Now let $\kappa> 0$ and define
$$
\mathcal{J}_\kappa := \{ (t,r)\in (0,T_e)\times (0,R_0)\ :\ J(t,r) \ge \kappa\}\ .
$$
Owing to the definition of $J$ there holds
\begin{equation}
r^{N-1} g^{p-1} = r^{N-1} a(g^2) g \le \kappa + r^{N-1} g^{p-1} \le c(r) F(u) = \delta r^\lambda u^\beta \;\;\text{ in }\;\; \mathcal{J}_\kappa\ . \label{Q12}
\end{equation}
A first consequence of \eqref{Q12} and the positivity of $\kappa$ is
that $r>0$ and $u>0$ in $\mathcal{J}_\kappa$, so that \eqref{Q9}
yields
\begin{equation}
\mathcal{R}_2 = - \lambda (\lambda-N)\delta (p-1) r^{\lambda-2} u^\beta g^{p-2} <0
\;\;\text{ in }\;\; \mathcal{J}_\kappa \label{Q9a}
\end{equation}
in view of $\lambda >N$ and $p \in (1,2]$. Moreover,
we deduce from \eqref{Q10}, \eqref{Q11}, and  \eqref{Q12} that, in $\mathcal{J}_\kappa$,
\begin{align}
\mathcal{R}_1 & \le 2 r^{\lambda-1} b'(g^2) g \left[ (N-1) r^{-\lambda} \delta r^\lambda u^\beta - \delta\lambda u^\beta \right] \nonumber \\
& \le - 2 (\lambda-N+1) \delta r^{\lambda-1} u^\beta b'(g^2) g\le 0 \label{Q14}
\end{align}
and
\begin{align*}
\mathcal{R}_3 & \le \beta \delta \left[ \frac{2(\lambda-N+1)}{r} \left( \delta r^{\lambda-N+1} u^\beta \right)^{(p-1-q)/(p-1)} \right. \\
& \qquad\qquad \left. + \frac{1-\beta}{u} \left[ \delta r^{\lambda-N+1} u^\beta \right]^{(p-q)/(p-1)} - (1-q) \right] r^\lambda u^{\beta-1} g^q \\
& \le \beta \delta \left[ 2(\lambda-N+1) \delta^{(p-1-q)/(p-1)} r^{((\lambda-N)(p-1-q)-q)/(p-1)} u^{\beta(p-1-q)/(p-1)} \right. \\
& \qquad\qquad+ (1-\beta) \delta^{(p-q)/(p-1)} r^{(\lambda-N+1)(p-q)/(p-1)} u^{(\beta(p-q)-(p-1))/(p-1)} \\
 & \qquad\qquad - (1-q) \Big] r^\lambda u^{\beta-1} g^q\ .
\end{align*}
In order to ensure $\mathcal{R}_3\le 0$ in $\mathcal{J}_\kappa$, we choose
$$
\beta := \frac{p-1}{p-q} \in (0,1) \;\;\text{ and }\;\; \lambda := N + \frac{q}{p-1-q} > N\ .
$$
Recalling that $q<p-1$, we use \eqref{zz4} and \eqref{zz1} to obtain that, in $\mathcal{J}_\kappa$,
\begin{align}
\mathcal{R}_3 & \le \beta \delta \left[ 2(\lambda-N+1) \delta^{(p-1-q)/(p-1)}  \|u_0\|_\infty^{(p-1-q)/(p-q)} \right. \nonumber \\
& \qquad\qquad \left. + (1-\beta) \delta^{(p-q)/(p-1)} R_0^{(p-q)/(p-q-1)} - (1-q) \right] r^\lambda u^{\beta-1} g^q \nonumber\\
& \le - \frac{\beta \delta (1-q)}{2} r^\lambda u^{\beta - 1} b(g^2) \le 0\ , \label{Q14b}
\end{align}
provided $\delta$ is chosen suitably small (depending on $\|u_0\|_\infty$ and $R_0$). Collecting \eqref{Q8},
\eqref{Q9a}, \eqref{Q14}, and \eqref{Q14b}, we have shown that
\begin{equation}
\partial_t J - a_1(g^2) \partial_r^2 J + \left( \frac{N-1}{r} a_1(g^2) - 2 b'(g^2)g \right) \partial_ r J < 0  \ , \qquad (t,r)\in \mathcal{J}_\kappa \ , \label{Q15}
\end{equation}
as soon as $\delta$ is suitably small.

Assume now for contradiction that the maximum of $J$ in
$[0,T_e]\times [0,R_0]$ exceeds $\kappa$. It is attained at some
point $(t_0,r_0)\in [0,T_e]\times [0,R_0]$ and it readily follows
from \eqref{zz14} that $r_0\in (0,R_0)$. In addition,
\eqref{init.data4} ensures that, for $r\in (0,R_0)$,
\begin{align*}
J(0,r) &= r^{\lambda} u_0(r)^{(p-1)/(p-q)} \left[ \delta - \left( \frac{1}{r^{1/(p-1-q)}} \frac{|\partial_r u_0(r)|}{u_0(r)^{1/(p-q)}}\right)^{p-1} \right] \\
& = r^{\lambda} u_0(r)^{(p-1)/(p-q)} \left[ \delta - \left( \frac{p-q}{p-1-q}\frac{1}{r^{1/(p-1-q)}} |\partial_r u_0^{(p-1-q)/(p-q)}(r)| \right)^{p-1} \right] \\
& \leq r^\lambda u_0(r)^{(p-1)/(p-q)} \left[ \delta -  \left(
\frac{p-q}{p-1-q} \delta_0 \right)^{p-1} \right] \leq 0
\end{align*}
by choosing $\delta>0$ even smaller. Consequently, $t_0>0$ and we
conclude that
$$
\partial_t J(t_0,r_0)\ge 0\ , \quad \partial_r J(t_0,r_0) = 0\ , \quad \partial_r^2 J(t_0,r_0)\le 0\ ,
$$
which contradicts \eqref{Q15}, since
$(t_0,r_0)\in\mathcal{J}_\kappa$. Therefore, $J\le \kappa$ in
$[0,T_e]\times [0,R_0]$. Since $\kappa$ is an arbitrary positive
number, we have shown that $J\le 0$ in $[0,T_e]\times [0,R_0]$ from
which \eqref{jay} follows.

We stress once more that this proof holds only at formal level as
the coefficients in the equation solved by $J$ may be singular. A
complete justification is postponed to the Appendix.
\end{proof}

The proof of Theorem~\ref{th.single} is now an easy consequence of the previous analysis as shown below.

\begin{proof}[Proof of Theorem~\ref{th.single}]
We wish to prove \eqref{incl.pos}, from which the single point
extinction follows readily. The left inclusion in \eqref{incl.pos}
has been proved in Lemma~\ref{lem.lower} and we are thus left with
the proof of the right inclusion. Fix $t\in (T_e/2,T_e)$. Taking
into account the radial symmetry and monotonicity of $u(t)$, we
deduce from Lemma~\ref{lem.J} that
$$
|\nabla u(t,x)|\geq\delta^{1/(p-1)} |x|^{1/(p-1-q)}\
u(t,x)^{1/(p-q)}\ , \qquad x\in B(0,R_0)\ ,
$$
hence, recalling \eqref{zz1},
\begin{equation}\label{interm21}
\frac{p-1-q}{p-q} \delta^{1/(p-1)} |x|^{1/(p-1-q)} \leq \left| \nabla u^{(p-1-q)/(p-q)}(t,x) \right|\ , \qquad x\in \mathcal{P}(t)\ .
\end{equation}
For $x\in\mathcal{P}(t)$, the monotonicity properties of $u(t)$
guarantee that $\varrho x\in\mathcal{P}(t)$ for all $\varrho\in
[0,1]$ and, for any $\varrho\in (0,1)$, we infer from \eqref{exp}
and \eqref{interm21} that
\begin{align*}
u^{(p-1-q)/(p-q)}(t,x) & = u^{(p-1-q)/(p-q)}(t,\varrho x) + \int_\varrho^1 \left\langle \nabla u^{(p-1-q)/(p-q)}(t,\sigma x) , x \right\rangle\ d\sigma \\
& =  u^{(p-1-q)/(p-q)}(t,\varrho x) - \int_\varrho^1 \left| \nabla u^{(p-1-q)/(p-q)}(t,\sigma x) \right| |x|\ d\sigma \\
& \le \|u(t)\|_\infty^{(p-1-q)/(p-q)} - C |x| \int_\varrho^1 |\sigma x|^{1/(p-1-q)}\ d\sigma \\
& \le \|u(t)\|_\infty^{(p-1-q)/(p-q)} - C |x|^{(p-q)/(p-1-q)} \left(
1 - \varrho^{(p-q)/(p-1-q)} \right)\ .
\end{align*}
We now apply Proposition~\ref{prop.ulbet} and let $\varrho\to 0$ to obtain that, given $\vartheta\in (0,1)$,
$$
u^{(p-q-1)/(p-q)}(t,x)\leq C_2(\vartheta)^{(p-q-1)/(p-q)}
(T_e-t)^{\vartheta(p-q-1)/(p-2q)}-C|x|^{(p-q)/(p-1-q)}\ .
$$
Since $x\in\mathcal{P}(t)$, this implies in particular that the right-hand side of the previous inequality is positive, that is,
$$
|x| < C(\vartheta) (T_e-t)^{\vartheta(p-q-1)^2/(p-q)(p-2q)}\ .
$$
Consequently, choosing $\vartheta=p/2$,
$$
\mathcal{P}(t)\subseteq B(0,\varrho_2(T_e-t)^{\nu}), \qquad
\nu=\frac{p(p-q-1)^2}{2(p-q)(p-2q)}\ ,
$$
for some $\varrho_2>0$. Since $t\in (T_e/2,T_e)$ is arbitrary, we have established the right inclusion in \eqref{incl.pos} and thereby completed the proof.
\end{proof}

%%%%%%%%%%%%%%%%%%%%
%%%%%%%%%%%%%%%%%%%%
\appendix
%%%%%%%%%%%%%%%%%%%%
\section{Proofs of Lemma~\ref{lem.J} and gradient estimates for $p=2$}
%%%%%%%%%%%%%%%%%%%%
%%%%%%%%%%%%%%%%%%%%
%\renewcommand{\thesection}{\Alph{section}}
%\setcounter{section}{1}
%\setcounter{theorem}{0}
%\setcounter{equation}{0}

In this technical section we provide a fully rigorous
proof of Lemma~\ref{lem.J}, as well as a gradient estimate for
solutions to \eqref{eq1} for $p=2$. The latter, besides of its
interest as an independent result, provides an essential technical
tool in the proofs of our main results, and complements the gradient
estimates in \cite[Theorem~1.3]{IL12}, valid for $p\in(p_c,2)$.

Lemma~\ref{lem.J} was proved in Section~\ref{sec.single} at a formal
level, presenting the essential calculations that give the ideas and
essence of the proof, but allowing us at that point, for the
simplicity of the exposition, to use results such as the maximum
principle (or comparison principle) that are not automatically
granted when we deal with singular coefficients. This is why we
include the rigorous proof of this lemma here. To this end, we
introduce a regularization of \eqref{eq1}, already successfully used
by the authors in \cite[Section~6]{IL12}, in order to avoid the
difficulties coming from the singularity at points where $\nabla
u=0$. Let then $u$ be a solution to the Cauchy problem
\eqref{eq1}-\eqref{eq2} associated to an initial condition $u_0$
satisfying the assumptions~(a)-(c) of Theorem~\ref{th.single}. We
recall that $u$ vanishes identically after a finite time $T_e$ and
that its positivity set $\mathcal{P}(t)$ is included in $B(0,R_0)$
for all $t\ge 0$, see \eqref{zz1}.

For $\e\in(0,1/2)$, we define:
\begin{equation}\label{approx}
a_{\e}(z):=(z+\e^2)^{(p-2)/2}, \qquad b_{\e}(z):=(z+\e^2)^{q/2},
\qquad z\geq0,
\end{equation}
and consider the following Cauchy problem:
\begin{equation}\label{approx2}
\left\{\begin{array}{ll}\partial_t \tilde{u}_{\e}-{\rm
div}\left(a_{\e}(|\nabla \tilde{u}_{\e}|^2)\nabla
\tilde{u}_{\e}\right)+b_{\e}\left(|\nabla \tilde{u}_{\e}|^2\right) -\e^q =0, \quad (t,x)\in(0,\infty)\times\real^N, \\
\tilde{u}_{\e}(0)= u_{0, \e} +\e^{\gamma},\qquad
x\in\real^N,\end{array}\right.
\end{equation}
where $\gamma\in(0,p/4)\cap(0,q/2)$ is a small positive parameter
such that $\gamma<\min\{p-1,1-q\}$ and $u_{0,\e}\in
C^{\infty}(\real^N)$ is a non-negative smooth approximation of the
initial condition $u_0$, in the sense that it converges to $u_0$
uniformly on compact sets in $\real^N$ and satisfies
\begin{equation}\label{approx3}
0 \le u_{0,\e} \leq\|u_0\|_{\infty}, \qquad \|\nabla
u_{0,\e}\|_{\infty}\leq(1+C(u_0)\e)\|\nabla u_0\|_{\infty}.
\end{equation}
It is proved in \cite{BL99} and
\cite[Section 6]{IL12} that \eqref{approx2} has a unique classical
solution $\tilde{u}_\e$ and that, as $\e\to 0$, it is an
approximation of the solution $u$ to \eqref{eq1}-\eqref{eq2} with
initial condition $u_0$ in the following sense:
\begin{equation}\label{conv}
u(t,x)=\lim\limits_{\e\to0} \tilde{u}_{\e}(t,x), \qquad \nabla
u(t,x)=\lim\limits_{\e\to0}\nabla \tilde{u}_{\e}(t,x),
\end{equation}
for almost every $(t,x)\in (0,\infty)\times\real^N$, the first
convergence being actually uniform in compact sets of
$(0,\infty)\times\real^N$. In addition, if $u_0$ is radially
symmetric and non-increasing, then $u_{0,\e}$ can be chosen to be
radially symmetric and non-increasing as well, so that $x\mapsto
\tilde{u}_\e(t,x)$ is radially symmetric and non-increasing for any
$t\ge 0$ and $\e\in(0,1/2)$. We next define
\begin{equation*}
u_\e(t,x) := \tilde{u}_\e(t,x) - \e^q t\ , \qquad (t,x)\in (0,\infty)\times\real^N\ ,
\end{equation*}
and observe that the comparison principle and \eqref{approx2} imply that
\begin{equation}
u_\e(t,x) \ge \e^\gamma - \e^q t \ge \frac{\e^\gamma}{2}\ , \qquad (t,x)\in (0,\tau_\e)\times \real^N\ , \label{approx2c}
\end{equation}
with $\tau_\e := \e^{\gamma-q}/2$ and that $u_\varepsilon$ solves
\begin{equation}\label{approx2d}
\left\{\begin{array}{ll}\partial_t u_{\e}-{\rm
div}\left(a_{\e}(|\nabla u_{\e}|^2)\nabla
u_{\e}\right)+b_{\e}\left(|\nabla u_{\e}|^2\right) =0, \quad (t,x)\in(0,\infty)\times\real^N, \\
u_{\e}(0)= u_{0,\e} +\e^{\gamma},\qquad
x\in\real^N\ .\end{array}\right.
\end{equation}
Since $\tau_\varepsilon\to\infty$ as $\varepsilon\to 0$, we may
assume that $\varepsilon$ is taken sufficiently small to ensure
$T_e\le \tau_\varepsilon$. With these approximations in mind, we are
ready to give the complete proof of Lemma~\ref{lem.J} as well as
that of the gradient estimate \eqref{grad.est1} for $p=2$.

%%%%%%%%%%%%%%%%%%%%
%%%%%%%%%%%%%%%%%%%%
\subsection{Proof of Lemma~\ref{lem.J}}
%%%%%%%%%%%%%%%%%%%%
%%%%%%%%%%%%%%%%%%%%

Owing to the $C^1$-smoothness of $u_0$, the gradient
convergence in \eqref{conv} is also uniform on compact
subsets of $\real^N$. Consequently,
\begin{equation}
m_\e := \e + \|u_{0,\e} - u_0 \|_{C^1(B(0,R_0))} \mathop{\longrightarrow}_{\e\to 0} 0\ . \label{X1}
\end{equation}
Introducing
$$
r_\e := \min \{ r \in [0,R_0) \: : \: u_0 (r) \le m_\e^{1/4} \}\ ,
$$
the properties of $u_0$ assumed in Theorem~\ref{th.single} and \eqref{X1} ensure that there is $\e_0 \in (0, 1/2)$ such that
\begin{equation}
r_\e > s_\e := m_\e^{(p-1-q)/4(p-q)} \;\; \text{ and }\;\;
m_\e\in (0,1) \;\;\text{ for any }\;\; \e \in (0,\e_0)\ .\label{X2}
\end{equation}
We then infer from the radial monotonicity of $u_0$, \eqref{init.data4}, \eqref{X1}, and \eqref{X2} that, for $r\in [s_\e,r_\e]$,
\begin{align*}
|\partial_r u_{0,\e}(r)| & \ge |\partial_r u_0(r)| - m_\e \\
& \ge \delta_0 r^{1/(p-1-q)} u_0(r)^{1/(p-q)} - m_\e \\
& \ge \frac{\delta_0}{2} r^{1/(p-1-q)} u_{0,\e}^{1/(p-q)}(r) + \frac{\delta_0}{2} r^{1/(p-1-q)} \left( u_0^{1/(p-q)} - u_{0,\e}^{1/(p-q)} \right)(r) \\
& \qquad + \frac{\delta_0}{2} m_\e^{1/2(p-q)} - m_\e \\
& \ge \frac{\delta_0}{2} r^{1/(p-1-q)} u_{0,\e}^{1/(p-q)}(r) - \frac{\delta_0}{2} R_0^{1/(p-1-q)} m_\e^{1/(p-q)} + \frac{\delta_0}{2} m_\e^{1/2(p-q)} - m_\e \\
& \ge \frac{\delta_0}{2} r^{1/(p-1-q)} u_{0,\e}^{1/(p-q)}(r) \\
& \qquad + m_\e^{1/2(p-q)} \left( \frac{\delta_0}{2} -
\frac{\delta_0}{2}  R_0^{1/(p-1-q)} m_\e^{1/2(p-q)} -
m_\e^{(2(p-q)-1)/ 2(p-q)} \right)\ .
\end{align*}
Since $p-q>1$, we infer from \eqref{X1} and the above inequality
that, taking $\e_0$ smaller if necessary, there holds
\begin{equation}\label{approx3a}
|\partial_r u_{0,\e} (r)| \ge \frac{\delta_0}{2} r^{1/(p-q-1)} u_{0,\e} (r)^{1/(p-q)}\ , \qquad r \in [s_\e, r_\e]\ , \qquad \e\in (0,\e_0) \ .
\end{equation}
Now fix $\e \in (0, \e_0)$. Recalling that $a_\varepsilon$ and
$b_\varepsilon$ are given by \eqref{approx}, we define
$$
a_{1,\varepsilon}(z) := 2 z a_\varepsilon'(z) + a_\varepsilon(z)\ , \qquad z\ge 0\ ,
$$
and the auxiliary function
$$
J_\varepsilon(t,r) := r^{N-1} a_\varepsilon(|\partial_r u_\varepsilon(t,r)|^2) \partial_r u_\varepsilon(t,r) + c(r) F(u_\varepsilon(t,r))\ , \qquad (t,r) \in (0,T_e)\times (0,R_0)\ .
$$
Since $u_\e$ solves \eqref{approx2d} and
\begin{align*}
(p-1) a_\varepsilon(z) \le \ a_{1,\varepsilon}(z) & = (z+\varepsilon^2)^{(p-4)/2} [(p-1)z+\varepsilon^2] \le a_\varepsilon(z) \ , \\
b_\varepsilon(z) - 2zb_\varepsilon'(z) & = (z+\varepsilon^2)^{(q-2)/2} [(1-q) z + \varepsilon^2] \ge (1-q) b_\varepsilon(z)
\end{align*}
for $z\ge 0$, we may perform the same computations as in the proof of Lemma~\ref{lem.J} with $(a_\varepsilon,b_\varepsilon)$ instead of $(a,b)$ and derive the analogue of \eqref{Q8}:
\begin{equation}
\partial_t J_\varepsilon - a_{1,\varepsilon}(g_\varepsilon^2) \partial_r^2 J_\varepsilon + \left( \frac{N-1}{r} a_{1,\varepsilon}(g_\varepsilon^2) - 2 b_\varepsilon'(g_\varepsilon^2)g_\varepsilon \right) \partial_ r J_\varepsilon \le \sum_{i=1}^3 \mathcal{R}_{i,\varepsilon}\ , \label{Q8eps}
\end{equation}
where $g_\varepsilon := - \partial_r u_\varepsilon \ge 0$ and
\begin{align*}
\mathcal{R}_{1,\varepsilon} & := 2 \left[ (N-1) r^{N-2} a_\varepsilon(g_\varepsilon^2) g_\varepsilon - c'(r) F(u_\varepsilon) \right] b_\varepsilon'(g_\varepsilon^2) g_\varepsilon\ , \\
\mathcal{R}_{2,\varepsilon} & := \left[ \frac{N-1}{r} c'(r) - c''(r) \right] F(u_\varepsilon) a_{1,\varepsilon}(g_\varepsilon^2) \ , \\
\mathcal{R}_{3,\varepsilon} & := 2 \left[ c'(r) - \frac{N-1}{r} c(r) \right] F'(u_\varepsilon) a_{1,\varepsilon}(g_\varepsilon^2) g_\varepsilon - c(r) F''(u_\varepsilon) a_{1,\varepsilon}(g_\varepsilon^2) g_\varepsilon^2 \\
& \qquad\qquad - (1-q) c(r) F'(u_\varepsilon) b_\varepsilon(g_\varepsilon^2)\ .
\end{align*}
Observe that the positivity \eqref{approx2c} of $u_\e$ guarantees
that $F'(u_\e)$ and $F''(u_\e)$ are well-defined, even if $F$ is not
twice differentiable at zero. As in the proof of Lemma~\ref{lem.J}
we choose
$$
c(r) = r^\lambda\ , \quad r\ge 0\ , \qquad F(z) = \delta z^\beta\ , \quad z\ge 0\ ,
$$
where $\delta>0$ is to be determined and
$$
\lambda = N + \frac{q}{p-1-q} > N  \ , \qquad \beta = \frac{p-1}{p-q}\in (0,1)\ .
$$
With this choice,
\begin{equation}
\mathcal{R}_{2,\varepsilon} \le - \lambda (\lambda-N)  \delta (p-1) r^{\lambda-2} u_\varepsilon^\beta a_{\varepsilon}(g_\varepsilon^2) \ ,  \label{Q9eps}
\end{equation}
while
\begin{equation}
\mathcal{R}_{1,\varepsilon} = 2 r^{\lambda-1} b_\varepsilon'(g_\varepsilon^2) g_\varepsilon \left[ (N-1) r^{N-1-\lambda} a_\varepsilon(g_\varepsilon^2) g_\varepsilon - \delta \lambda u_\varepsilon^\beta \right] \label{Q10eps}
\end{equation}
and
\begin{align}
\mathcal{R}_{3,\varepsilon} & \le \beta \delta \left[ \frac{2(\lambda-N+1)}{r} \left( g_\varepsilon^2 + \varepsilon^2 \right)^{(p-1-q)/2} + \frac{1-\beta}{u_\varepsilon} \left( g_\varepsilon^2 + \varepsilon^2 \right)^{(p-q)/2} \right. \nonumber \\
& \qquad \qquad \qquad \qquad- (1-q) \Big] r^\lambda u_\varepsilon^{\beta-1} b_\varepsilon(g_\varepsilon^2)\ . \label{Q11eps}
\end{align}
Now let $\kappa> 0$ and define
$$
\mathcal{J}_{\kappa,\varepsilon} := \{ (t,r)\in (0,T_e)\times (0,R_0)\ :\ J_\varepsilon(t,r) \ge \kappa\}\ .
$$
Then
\begin{equation*}
\kappa + r^{N-1} a_\varepsilon(g_\varepsilon^2) g_\varepsilon \le c(r) F(u_\varepsilon) = \delta r^\lambda u_\varepsilon^\beta \;\;\text{ in }\;\; \mathcal{J}_{\kappa,\varepsilon}\ ,
\end{equation*}
from which we deduce that, if $\kappa\ge R_0^{N-1} \varepsilon^{p-1}$, then
\begin{equation}
r^{N-1} (g_\varepsilon^2 + \varepsilon^2 )^{(p-1)/2} \le r^{N-1} a_\varepsilon(g_\varepsilon^2) g_\varepsilon + R_0^{N-1} \varepsilon^{p-1} \le \delta r^\lambda u_\varepsilon^\beta  \;\;\text{ in }\;\; \mathcal{J}_{\kappa,\varepsilon}\ . \label{Q12eps}
\end{equation}
This inequality implies in particular that $r>0$ and $u_\varepsilon
>0$ in $\mathcal{J}_{\kappa,\varepsilon}$ and \eqref{Q9eps} yields
\begin{equation}
\mathcal{R}_2 \le - \lambda (\lambda-N)\delta (p-1) r^{\lambda-2} u_\varepsilon^\beta
\left( g_\varepsilon^2 + \varepsilon^2 \right)^{(p-2)/2} <0
\;\;\text{ in }\;\; \mathcal{J}_{\kappa,\varepsilon} \label{Q9aeps}
\end{equation}
in view of $\lambda >N$ and $p \in (1,2]$. We then infer from \eqref{Q10eps}, \eqref{Q11eps}, and \eqref{Q12eps} that, in $\mathcal{J}_{\kappa,\varepsilon}$,
\begin{equation}
\mathcal{R}_{1,\varepsilon} \le - 2 (\lambda-N+1) \delta r^{\lambda-1} u_\varepsilon^\beta b_\varepsilon'(g_\varepsilon^2) g_\varepsilon\le 0 \label{Q14eps}
\end{equation}
and
\begin{align}
\mathcal{R}_{3,\varepsilon} & \le \beta \delta \left[ 2(\lambda-N+1) \delta^{(p-1-q)/(p-1)} \left( \|u_0\|_\infty +\varepsilon^\gamma\right)^{(p-1-q)/(p-q)} \right. \nonumber\\
& \qquad\qquad \left. + (1-\beta) \delta^{(p-q)/(p-1)} R_0^{(p-q)/(p-1-q)} - (1-q) \right] r^\lambda u_\varepsilon^{\beta-1} b_\varepsilon(g_\varepsilon^2) \nonumber\\
& \le - \frac{\beta \delta (1-q)}{2} r^\lambda u_\varepsilon^{\beta -1} b_\varepsilon(g_\varepsilon^2)
\le 0  \ , \label{Q14beps}
\end{align}
provided $\delta$ is chosen suitably small (depending on $\|u_0\|_\infty$ and $R_0$) and independent of $\varepsilon\in
(0,\e_0)$ as $\|u_0\|_\infty +\varepsilon^\gamma \le \|u_0\|_\infty +1$. Collecting \eqref{Q8eps}, \eqref{Q9aeps}, \eqref{Q14eps}, and \eqref{Q14beps}, we end up with
\begin{equation}
\partial_t J_\varepsilon - a_{1,\varepsilon}(g_\varepsilon^2) \partial_r^2 J_\varepsilon + \left( \frac{N-1}{r} a_{1,\varepsilon}(g_\varepsilon^2) - 2 b_\varepsilon'(g_\varepsilon^2)g_\varepsilon \right) \partial_ r J_\varepsilon < 0 \qquad\mbox{for } (t,r)\in \mathcal{J}_{\kappa,\varepsilon} \ ,
\label{Q15eps}
\end{equation}
this inequality being true only for $\kappa\ge R_0^{N-1} \varepsilon^{p-1}$.

Next, introducing
$$
M_\varepsilon := \sup_{t\in [0,T_e]} u_\varepsilon(t,R_0)\ ,
$$
we infer from the monotonicity of $r\mapsto u_\varepsilon(t,r)$ that
\begin{equation}
J_\varepsilon(t,0)=0\ , \qquad J_\varepsilon(t,R_0) \le \delta R_0^\lambda M_\varepsilon^\beta\ , \quad t\in [0,T_e]\ . \label{Q16eps}
\end{equation}
In addition, given $r\in (0,R_0)$,
$$
J_\varepsilon(0,r) \le r^{N-1} a_\varepsilon(|\partial_r u_{0,\e}(r)|^2) \partial_r u_{0,\e}(r) + \delta r^\lambda u_{0,\e}^\beta(r) + \delta R_0^\lambda \varepsilon^{\gamma\beta} \ .
$$
As $\partial_r  u_{0,\e}(r) \le 0$, we obtain from \eqref{approx3} and \eqref{X1} that
\begin{equation}\label{Q17eps}
  J_\varepsilon(0,r) \le \delta s_\e^\lambda \| u_0 \|_\infty^\beta + \delta R_0^\lambda \varepsilon^{\gamma\beta}, \qquad r \in (0,s_\e) \ ,
\end{equation}
as well as
\begin{equation}\label{Q18eps}
  J_\varepsilon(0,r) \le \delta R_0^\lambda\left( m_\e + m_\e^{1/4} \right)^\beta + \delta R_0^\lambda \varepsilon^{\gamma\beta}, \qquad r \in (r_\e, R_0) \ .
\end{equation}
For $r \in [s_\e, r_\e]$, we now divide the analysis into
two regions with respect to the magnitude of $|\partial_r
u_{0,\e}(r)|$. Either $|\partial_r u_{0,\e}(r)| \le \varepsilon$ and
we deduce from \eqref{approx3a} that
$$
u_{0,\e}(r)^{1/(p-q)}\leq\frac{2\e}{\delta_0}r^{-1/(p-q-1)} \ .
$$
Thus, taking also into account that $\partial_r u_{0,\e}(r)\leq0$ for any $r\geq0$, we realize that
\begin{equation}\label{Q19eps}
\begin{split}
J_\varepsilon(0,r) & \le \delta r^\lambda (2 \varepsilon)^{p-1} \delta_0^{1-p} r^{-(p-1)/(p-1-q)} + \delta R_0^\lambda \varepsilon^{\gamma\beta} \\
& \le \delta \delta_0^{1-p} R_0^{N-1} (2 \varepsilon)^{p-1} + \delta R_0^\lambda \varepsilon^{\gamma\beta}\ .
\end{split}
\end{equation}
Or $|\partial_r u_{0,\e}(r)| > \varepsilon$ which implies that
$$
a_\varepsilon(|\partial_r u_{0,\e}(r)|^2) \ge 2^{(p-2)/2} |\partial_r u_{0,\e}(r)|^{p-2}\ .
$$
Therefore, using again \eqref{approx3a},
\begin{equation}\label{Q20eps}
\begin{split}
J_\varepsilon(0,r) & \le \delta r^\lambda u_{0,\e}(r)^\beta - 2^{(p-2)/2} r^{N-1} |\partial_r u_{0,\e}(r)|^{p-1} + \delta R_0^\lambda \varepsilon^{\gamma\beta} \\
& \le \left( \delta - 2^{-p/2} \delta_0^{p-1} \right) r^\lambda u_{0,\e}(r)^\beta + \delta R_0^\lambda \varepsilon^{\gamma\beta} \\
& \le \delta R_0^\lambda \varepsilon^{\gamma\beta}\ ,
\end{split}
\end{equation}
provided $\delta < 2^{-p/2} \delta_0^{p-1}$. In view of \eqref{X2}
and \eqref{Q17eps}-\eqref{Q20eps} we have thus shown that, if
$\delta$ is sufficiently small,
$$
J_\varepsilon(0,r) \le \delta \delta_0^{1-p} R_0^{N-1}
(2\varepsilon)^{p-1} + \delta \| u_0 \|_\infty^\beta s_\e^\lambda +
\delta R_0^\lambda \left((2m_\e)^{\beta/4} +
\varepsilon^{\gamma\beta} \right) \ , \qquad r\in (0,R_0)\ .
$$
Consequently, if $\delta$ is sufficiently small and
\begin{align*}
\kappa = \kappa_\varepsilon & := \delta \delta_0^{1-p} R_0^{N-1}
(2\varepsilon)^{p-1} +
\delta \| u_0 \|_\infty^\beta s_\e^\lambda + \delta R_0^\lambda \left( (2m_\e)^{\beta/4} + \varepsilon^{\gamma\beta} \right) \\
& \qquad + R_0^{N-1} \varepsilon^{p-1} + \delta R_0^\lambda M_\varepsilon^\beta \ ,
\end{align*}
then the parabolic boundary $\{0\}\times (0,R_0)$ and $[0,T_e)
\times \{0,R_0\}$ of $(0,T_e)\times (0,R_0)$ contains no point in
$\mathcal{J}_{\kappa_\varepsilon,\varepsilon}$. Recalling
\eqref{Q15eps} we may then argue as in the proof of
Lemma~\ref{lem.J} to conclude that
\begin{equation}
J_\varepsilon\le \kappa_\varepsilon \;\;\text{ in }\;\; (0,T_e)\times (0,R_0)\ . \label{Q100}
\end{equation}

\noindent To complete the proof, we observe that $M_\varepsilon\to
0$ as $\varepsilon\to 0$ due to the uniform convergence \eqref{conv}
and the vanishing of $u$ on $(0,T_e)\times\partial B(0,R_0)$.
Combining this fact with \eqref{X1} and \eqref{X2} yields
$$
\lim_{\varepsilon\to 0} \kappa_\varepsilon =0\ ,
$$
and we may let $\varepsilon\to 0$ in \eqref{Q100} and use \eqref{conv} and \eqref{X1} to obtain the expected result.

%%%%%%%%%%%%%%%%%%%%
%%%%%%%%%%%%%%%%%%%%
\subsection{Proof of \eqref{grad.est1} for $p=2$.}
%%%%%%%%%%%%%%%%%%%%
%%%%%%%%%%%%%%%%%%%%

Finally, we prove the gradient estimate \eqref{grad.est1} for $p=2$.

%%%%%%%%%%%%%%%%%%%
\begin{proposition}\label{prop.gradest}
 Consider an initial condition $u_0$ satisfying \eqref{hypIC} and denote the corresponding solution to \eqref{eq1}-\eqref{eq2} by $u$.
    Assume further that $p=2$ and $q \in (0,1)$. Then there is $C_1>0$ depending only on $q$ such that
    \begin{equation}\label{grad.est1a}
    \left|\nabla u^{(1-q)/(2-q)}(t,x)\right|\leq
    C_1\left(1+\|u_0\|_{\infty}^{(1-q)/(2-q)}t^{-1/2}\right),
  \end{equation}
  for $(t,x)\in (0,\infty)\times\real^N$.
\end{proposition}
%%%%%%%%%%%%%%%%%%%%

\begin{proof}
  We fix $\e \in (0, 1/2)$ and denote the classical solution to \eqref{approx2} by $\tilde{u}_\e$. Observe that $a_\e \equiv 1$ due to $p=2$. In view of \eqref{approx3}, the comparison principle implies
    \begin{equation}\label{pge1}
      \e^\gamma \le \tilde{u}_\e (t,x) \le \|u_0 \|_\infty + \e^\gamma, \qquad (t,x) \in [0,\infty) \times \real^N \ .
    \end{equation}
    We further set
    $$
      f(\xi) := \frac{1-q}{2-q} \xi^{(2-q)/(1-q)}, \; \xi \ge 0, \qquad v_\e := f^{-1} (\tilde{u}_\e), \qquad w_\e := |\nabla v_\e|^2,
    $$
and note that $f \in C^2([0,\infty)) \cap C^\infty ((0,\infty))$ is
strictly increasing. Hence, according to \cite[formula~(10)]{BL99},
we have
    \begin{equation}\label{pge2}
      \cp_\e w_\e \le 2 \left( \frac{f^{\prime \prime}}{f^\prime} \right)^\prime (v_\e) w_\e^2
        - 2 \left( \frac{f^{\prime \prime}}{(f^\prime)^2} \right) (v_\e) \Theta_\e \left( (f^\prime)^2 (v_\e) w_\e \right) w_\e
    \end{equation}
    in $(0,\infty) \times \real^N$, where
    \begin{align*}
      & \cp_\e w_\e := \partial_t w_\e - \Delta w_\e + 2 \left( f^\prime (v_\e) b_\e^\prime \left( (f^\prime)^2 (v_\e) w_\e \right)
        - \left( \frac{f^{\prime \prime}}{f^\prime} \right) (v_\e) \right) \nabla v_\e \cdot \nabla w_\e \ , \\
        &\Theta_\e (\xi) := 2\xi b_\e^\prime (\xi) - b_\e(\xi) + \e^q \ , \qquad \xi \ge 0 \ .
    \end{align*}
    Since $q \in (0,1)$, we have
    \begin{align*}
      \Theta_\e (\xi) & = q \xi (\xi + \e^2)^{(q-2)/2} - (\xi + \e^2)^{q/2} + \e^q \\
      &= -(1-q) (\xi + \e^2)^{q/2} - q \e^2 (\xi + \e^2)^{(q-2)/2} + \e^q \\
        & \ge -(1-q) (\xi + \e^2)^{q/2} + (1-q) \e^q \\
        & \ge -(1-q) \xi^{q/2} \ , \qquad \xi \ge 0 \ .
    \end{align*}
    Hence, \eqref{pge2}, the choice of $f$, the nonnegativity of $w_\e$, and Young's inequality imply
    \begin{equation}\label{pge3}
    \begin{split}
      \cp_\e w_\e
        &\le - \frac{2}{(1-q) v_\e^2} w_\e^2 + 2 \cdot \frac{1}{1-q} v_\e^{(q-2)/(1-q)} \cdot (1-q) \left( v_\e^{2/(1-q)} w_\e \right)^{q/2} w_\e\\
        &\le  - \frac{2}{(1-q) v_\e^2} w_\e^2 + \frac{2}{(1-q)v_\e^2} w_\e^{1+ q/2}\\
        &\le  \frac{2}{(1-q) v_\e^2} \left(- w_\e^2 + \frac{2+q}{4} w_\e^2 + \frac{2-q}{4} \right)\\
        &=  - \frac{2-q}{2(1-q) v_\e^2} \left( w_\e^2 -1 \right) \ .
    \end{split}
    \end{equation}
    Considering next
    $$
      W(t) := 1 + \frac{a}{t}, \; t >0, \qquad\mbox{with } \;
        a := 2 \left( \frac{1-q}{2-q} \right)^{q/(2-q)} (\|u_0\|_\infty +1)^{2(1-q)/(2-q)},
    $$
    noticing that \eqref{pge1} implies
    $$
      v_\e = \left( \frac{2-q}{1-q} \tilde{u}_\e \right)^{(1-q)/(2-q)} \le \left( \frac{2-q}{1-q}  (\|u_0\|_\infty +1) \right)^{(1-q)/(2-q)},
    $$
    and using the fact that $a>0$, we obtain
    \begin{align*}
     \cp_\e W + \frac{2-q}{2(1-q) v_\e^2} \left( W^2 -1 \right)
        &\ge -\frac{a}{t^2} + \frac{2-q}{2(1-q) v_\e^2} \cdot \frac{a^2}{t^2} \\
        &\ge \frac{a}{t^2} \left[ -1 + \frac{a}{2} \left( \frac{2-q}{1-q} \right)^{q/(2-q)}
        (\|u_0\|_\infty +1)^{-2(1-q)/(2-q)} \right] \\
        & \ge 0
    \end{align*}
    for any $t>0$. Since $w_\e (0,x) < W(0) = \infty$ for $x \in \real^N$, we deduce from \eqref{pge3} and
    the comparison principle that
    $$
    \left( \frac{2-q}{1-q} \right)^{(1-q)/(2-q)} \left| \nabla \tilde{u}_\e^{(1-q)/(2-q)} (t,x) \right|
    = |\nabla v_\e (t,x) | = w_\e^{1/2} (t,x) \le \left(1+ \frac{a}{t} \right)^{1/2}
    $$
    in $(0,\infty) \times \real^N$. Letting $\e \searrow 0$ and recalling \eqref{conv}, we end up with \eqref{grad.est1a}.
\end{proof}

%%%%%%%%%%%%%%%%%%%%
%%%%%%%%%%%%%%%%%%%%
\section*{Acknowledgments}
%%%%%%%%%%%%%%%%%%%%
%%%%%%%%%%%%%%%%%%%%

R. I. is partially supported by the Spanish project MTM2012-31103
and the Severo Ochoa Excellence project SEV-2011-0087.
C. S. acknowledges the support of the Carl Zeiss foundation. Part of
this work was done while C. S. held a one month invited position at
the Institut de Mathématiques de Toulouse, and during visits of R.
I. to the Institut de Mathématiques de Toulouse. Both authors would
like to express their gratitude for the support and hospitality.

%%%%%%%%%%%%%%%%%%%%
%%%%%%%%%%%%%%%%%%%%
\bibliographystyle{plain}

\begin{thebibliography}{}

\end{thebibliography}


\begin{thebibliography}{1}

\bibitem{Ab98} U. G.~Abdullaev, \emph{Instantaneous shrinking of the support of solutions to a nonlinear degenerate parabolic equation}, Math. Notes, \textbf{63} (1998), no. 3, 285--292.

\bibitem{ATU04} D. Andreucci, A.~F.~Tedeev, and M. Ughi,
\emph{The Cauchy problem for degenerate parabolic equations with
source and damping}, Ukrainian Math. Bull., \textbf{1} (2004),
1--23.

\bibitem{BSW02} M.~Ben-Artzi, Ph.~Souplet, and F.B.~Weissler, \emph{The local theory for viscous Hamilton-Jacobi equations in Lebesgue spaces},
J. Math. Pures Appl., \textbf{81} (2002), 343--378.

\bibitem{BKaL04}
S. Benachour, G. Karch, and Ph. Lauren\ced{c}ot, \emph{Asymptotic
profiles of solutions to viscous Hamilton-Jacobi equations}, J.
Math. Pures Appl., \textbf{83} (2004), 1275--1308.

\bibitem{BL99}
S. Benachour and Ph. Lauren\ced{c}ot, \emph{Global solutions to
viscous Hamilton-Jacobi equations with irregular initial data},
Comm. Partial Differential Equations, \textbf{24} (1999), no. 11-12,
1999--2021.

\bibitem{BLS01}
S. Benachour, Ph. Lauren\ced{c}ot, and D. Schmitt, \emph{Extinction
and decay estimates for viscous Hamilton-Jacobi equations in
$\real^N$}, Proc. Amer. Math. Soc., \textbf{130} (2001), no. 4,
1103--1111.

\bibitem{BLSS02}
S. Benachour, Ph. Lauren\ced{c}ot, D. Schmitt, and Ph. Souplet,
\emph{Extinction and non-extinction for viscous Hamilton-Jacobi
equations in $\real^N$}, Asympt. Anal., \textbf{31} (2002),
229--246.

\bibitem{BRV97}
S. Benachour, B. Roynette, and P. Vallois, \emph{Asymptotic
estimates of solutions of $u_t-\frac{1}{2}\Delta u=-|\nabla u|$ in
$\real_{+}\times\real^{d}$, $d\geq2$}, J. Funct. Anal., \textbf{144}
(1997), 301--324.

\bibitem{BVD13} M.-F.~Bidaut-V\'eron and N.A.~Dao, \emph{$L^\infty$
estimates and uniqueness results for nonlinear parabolic equations
with gradient absorption terms}, Nonlinear Anal., \textbf{91} (2013), 121--152.

\bibitem{BGK04} P. Biler, M. Guedda, and G. Karch, \emph{Asymptotic
properties of solutions of the viscous Hamilton-Jacobi equation}, J.
Evolution Equations, \textbf{4} (2004), 75--97.

\bibitem{BU}
M. Borelli and M. Ughi, \emph{The fast diffusion equation with
strong absorption: the instantaneous shrinking phenomenon},
Rend. Ist. Mat. Univ. Trieste, \textbf{26} (1994), 109--140.

\bibitem{De08}
S.P.~Degtyarev, \emph{Conditions for instantaneous support shrinking
and sharp estimates for the support of the solution of the Cauchy
problem for a doubly non-linear parabolic equation with absorption},
Sb. Math., \textbf{199} (2008), no. 4, 511--538.

\bibitem{DPS}
M. del Pino and M. Sa\'ez, \emph{Asymptotic description of vanishing
in a fast-diffusion equation with absorption}, Differential Integral
Equations, \textbf{15} (2002), no. 8, 1009--1023.

\bibitem{EK}
L. C. Evans and B. F. Knerr, \emph{Instantaneous shrinking of the
support of nonnegative solutions to certain nonlinear parabolic
equations and variational inequalities}, Illinois J. Math.,
\textbf{23} (1) (1979), 153--166.

\bibitem{FGV}
R. Ferreira, V. A. Galaktionov, and J. L. V\'azquez, \emph{Uniqueness
of asymptotic profiles for an extinction problem}, Nonlinear Anal.,
\textbf{50} (2002), no. 4, 495--507.

\bibitem{FV}
R. Ferreira and J. L. V\'azquez, \emph{Extinction behaviour for fast
diffusion equations with absorption}, Nonlinear Anal., \textbf{43}
(2001), no. 8, 943--985.

\bibitem{FH}
A. Friedman and M. A. Herrero, \emph{Extinction properties of
semilinear heat equations with strong absorption}, J. Math. Anal. Appl., \textbf{124} (1987), no. 2, 530--546.

\bibitem{FML}
A. Friedman and B. McLeod, \emph{Blow-up of positive solutions of
semilinear heat equations}, Indiana Univ. Math. J., \textbf{34}
(1985), no. 2, 425--447.

\bibitem{GL07} Th. Gallay and Ph. Lauren\c cot,
\emph{Asymptotic behavior for a viscous Hamilton-Jacobi equation
with critical exponent}, Indiana Univ. Math. J., \textbf{56} (2007),
459--479.

\bibitem{GGK03} B. H.~Gilding, M.~Guedda, and R.~Kersner, \emph{The Cauchy problem for $u_t = \Delta u + |\nabla u|^q$},
J. Math. Anal. Appl., \textbf{284} (2003), 733--755.

\bibitem{GK90}
B. H. Gilding and R. Kersner, \emph{Instantaneous shrinking in
nonlinear diffusion-convection}, Proc. Amer. Math. Soc, \textbf{109}
(1990), no. 2, 385--394.

\bibitem{Gi05} B. H.~Gilding, \emph{The Cauchy problem for $u_t = \Delta u + |\nabla u|^q$, large-time behaviour}, J. Math. Pures Appl., \textbf{84} (2005), 753--785.

\bibitem{HV89}
M. A.~Herrero and J. J. L.~Vel\'azquez, \emph{On the dynamics of a semilinear heat equation with strong absorption}, Comm. Partial Differential Equations, \textbf{14} (1989), no. 12, 1653--1715.

\bibitem{HV92}
M. A.~Herrero and J. J. L.~Vel\'azquez, \emph{Approaching an extinction point in one-dimensional semilinear heat equations with strong absorption}, J. Math. Anal. Appl. \textbf{170} (1992), 353--381.

\bibitem{IL12}
R.~Iagar and Ph.~Lauren\ced{c}ot, \emph{Positivity, decay and
extinction for a singular diffusion equation with gradient
absorption}, J. Funct. Anal., \textbf{262} (2012), no. 7,
3186--3239.

\bibitem{K90}
A. S.~Kalashnikov, \emph{Conditions for the instantaneous
compactifications of carriers of solutions of semilinear parabolic equations and systems}, Math. Notes, \textbf{47} (1990), no.
1-2, 49--53.

\bibitem{K93}
A. S.~Kalashnikov, \emph{Quasilinear degenerate parabolic equations with singular lowest terms and growing initial data}, Differential Equations, \textbf{29} (1993), no. 6,
857--866.

\bibitem{KN92}
R. Kersner and F. Nicolosi, \emph{The nonlinear heat equation with
absorption: effects of variable coefficients}, J. Math. Anal. Appl.,
\textbf{170} (1992), no. 2, 551--566.

\bibitem{KS96}
R. Kersner and A. Shishkov, \emph{Instantaneous shrinking of the
support of energy solutions}, J. Math. Anal. Appl, \textbf{198}
(1996), no. 3, 729--750.

\bibitem{HJBook}
Ph. Lauren\ced{c}ot, \emph{Large time behavior for diffusive
Hamilton-Jacobi equations}, in Topics in Mathematical Modeling,
Lecture Notes, vol. 4, Jindrich Necas Center for Mathematical
Modeling, Praha, 2008.

\end{thebibliography}

%%%%%%%%%%%%%%%%%%%%
%%%%%%%%%%%%%%%%%%%%

\end{document}